\documentclass[12pt]{amsart}
\usepackage[latin1]{inputenc}   
\usepackage[T1]{fontenc}
\usepackage{amsmath}
\usepackage{amsthm}
\usepackage{thmtools}
\usepackage{amssymb}
\usepackage{amscd}
\usepackage{enumitem}
\usepackage{bm}
\theoremstyle{plain}
\usepackage{geometry}
\usepackage{caption}
\usepackage[pdftex,colorlinks]{hyperref}

\usepackage{tikz}
\usetikzlibrary{arrows}
\usetikzlibrary{decorations.markings}
\tikzstyle directed=[postaction={decorate,decoration={markings,
    mark=at position .65 with {\arrow{latex}}}}]

\def\Im{\operatorname{Im}}

\def\dim{\operatorname{dim}}

\def\Int{\operatorname{Int}}
\def\Id{\operatorname{Id}}
\def\Cut{\operatorname{Cut}}
\def\Ind{\operatorname{Ind}}
\def\ev{\operatorname{ev}}
\def\deter{\operatorname{det}}

\newcommand\norm[1]{\left \vert #1 \right \vert}
\newcommand\ens[2]{\left \{ #1 \vert #2 \right \}}

\newcommand \Z{\mathbb{Z}}
\newcommand \R {\mathbb{R}}
\newcommand \N {\mathbb{N}}
\newcommand \C {\mathbb{C}}
\newcommand \D {\mathbb{D}}
\newcommand \HP {\mathbb{H}}

\newcommand \T {\mathbb{T}}

\newtheorem{Theo}{Theorem}

\newtheorem{Ex}{Example}
\newtheorem{prop}{Proposition}
\newtheorem{lemma}{Lemma}
\newtheorem{coro}{Corollary}
\newtheorem{defi}{Definition}
\newtheorem{rk}{Remark}

\author{Alexandre Perrier}
\email{perrier@dms.umontreal.ca}
\title[Structure of $J$-holomorphic disks ]{Structure of $J$-holomorphic disks with immersed Lagrangian boundary conditions}

\begin{document}

\begin{abstract}
We explain how to generalize Lazzarini's structural Theorem from \cite{La11} to the case of curves with boundary on a given Lagrangian immersion. As a consequence of this result, we show that we can compute Floer homology with time-independent almost complex structures. We also give some applications as well as topics for future work.  	
\end{abstract}

\maketitle

\section{Introduction}
	
\subsection{Setting}
\label{subsection:historicalcontext}
Let $(M,\omega)$ be a symplectic manifold and $J \in \mathcal{J}(M,\omega)$ be a compatible almost complex structure. It is well known that any $J$-holomorphic curve $u : \Sigma \to M$ with $\Sigma$ a closed Riemann surface factors through a simple curve (see \cite[Proposition 2.5.1]{McDS12}).

Let $L \subset M$ be an embedded Lagrangian submanifold and
	\[u : (\D, \partial \D) \to (M, L) \]
 be a $J$-holomorphic disk satisfying $u(\partial \D) \subset L $. In general, it is not true that such a map factors through a branched cover to a simple curve. However there are results of Kwon-Oh (\cite{Oh97},\cite{KO00}) and Lazzarini (\cite{La00}, \cite{La11}) about the structure of such disks. 

Moduli spaces of disks with Lagrangian boundaries appear in the definitions of several differential complexes associated to Lagrangian embeddings such as the pearl complex (due to Biran-Cornea \cite{BC07}, \cite{BC09}) or Lagrangian intersection Floer homology in the monotone case (due to Oh, \cite{Oh93}, \cite{Oh93-2}). The results of Lazzarini and Kwon-Oh are essential to study the generic regularity of such moduli spaces.

\subsection{Main theorem}
\label{subsection:maintheorem}
In this paper, we shall explain how to adapt Lazzarini's result (\cite{La11}) to disks with corners whose boundaries lie in the image of a Lagrangian immersion. In this section, we provide the basic definitions of the objects we will consider.

From now on, we fix a connected symplectic manifold $(M^{2n}, \omega)$ and a Lagrangian immersion $i : L^n \looparrowright M$, with $L$ a closed (not necessarily connected) manifold such that 
\begin{enumerate}
	\item $i$ does not have triple points,
	\item the double points of $i$ are transverse.	 
\end{enumerate}
If this is satisfied, we say that $i$ is \emph{generic}.

Let us denote by $R = \ens{(p,q) \in L \times L}{i(p) = i(q)} $ the set of ordered double points of $i$ and by $i(R)$ their images. The hypotheses on $i$ imply that this is a finite subset.

Moreover, we fix a (smooth) compatible almost complex structure $J \in \mathcal{J}(M, \omega)$. We now explain what we mean by an almost complex curve with corners and boundary on $L$.

\begin{defi}
\label{defi:Curvewithboundary}
Let $S$ be a compact Riemann surface with boundary $\partial S$. 

A \emph{$J$-holomorphic curve with corners and boundary on $L$} is a continuous map 
\[u : (S, \partial S) \to (M, i(L) )\] 
which satisfies the following assumptions.
\begin{enumerate}[label=(\roman*)]
	\item There are $x_1 , \ldots, x_N \in \partial S$ with
	\[ \forall 1 \leqslant k \leqslant N, u(x_k) \in i(R) .\]
	\item There is a continuous map $\gamma : \partial S \backslash \{x_1, \ldots, x_N \} \to L$ such that
	\[ u_{\vert \partial S \backslash \{ x_1, \ldots, x_N \}} = i \circ \gamma .\]
	\item The map $\gamma$ does not extend to a continuous map $\partial S \to L$.
	\item The map $u$ is a smooth $J$-holomorphic curve on $S \backslash \{ x_1, \ldots , x_N \}$. 
\end{enumerate}
\end{defi}

\begin{rk}
\label{rk:defoftopologicalcurves}
\begin{enumerate}
	\item Keeping the notations of Definition \ref{defi:Curvewithboundary}, we call $x_1, \ldots, x_N$ the \emph{corner points} of the curve.
	\item We also consider maps $u : S \to M$ which satisfy the hypotheses $(i), (ii), (iii)$ without $(iv)$. We call such a map a \emph{topological} curve with corners.
\end{enumerate} 
\end{rk}

For a $J$-holomorphic curve $u : (S, \partial S) \to (M, i(L))$ with corners and boundary on $L$, a point $z \in \Int(S)$ is an \emph{injective point} if it satisfies
\[ d u_z \neq 0 , u^{-1}(u(z)) = \{ z\}. \]
We say that such a curve is \emph{simple} if the set of its injective points is dense. 

We can now state the main theorem of this paper.

\begin{Theo}
\label{Theo:Maintheorem}
Let $u : (\D, \partial \D) \to (M,i(L))$ be a non-constant $J$-holomorphic disk with corners, boundary on $L$ and finite energy (meaning $ \int u^* \omega < + \infty$ ).

There are simple finite-energy $J$-holomorphic disks $v_1, \ldots, v_N$ with corners, boundary on $L$ and natural integers $m_1, \ldots, m_N \in \N$ such that
\begin{enumerate}[label=(\roman*)]
	\item $\Im(u) = \cup_{k = 1 \ldots N} \Im(v_k)$
	\item In $H_2 \left (M, i(L) \right )$ we have
		\[ [u] = \sum_{k = 1}^N m_k [v_k] .\]	
\end{enumerate}
\end{Theo}

The proof of this is an adaptation of Lazzarini's proof to the case of immersed Lagrangians.
 
\subsection{Applications}
\label{subsection:Applications}
Assume that the complex dimension $n$ is greater than $3$. For a generic almost complex structure $J$, any finite-energy $J$-holmorphic disk with corners and boundary on $L$ is either simple or multiply covered. This follows from an adaptation of the proof of \cite[Proposition 5.15]{La11}.

\begin{coro}
\label{coro:Jholdisksaremultiplycovered}
Suppose $n \geqslant 3$. There is a second category subset $\mathcal{J}_{\text{reg}} (M, \omega, L) \subset \mathcal{J}(M, \omega)$ satisfying the following property.  

Let $J \in \mathcal{J}_{\text{reg}} (M, \omega)$ and $u : (\D, \partial \D) \to (M,i(L) )$ be a non-constant finite-energy $J$-holomorphic disk with corners and boundary on $L$. Then there exist 
\begin{enumerate}[label=(\roman*)]
	\item a holomorphic map $p : (\D, \partial \D) \to (\D, \partial \D)$ with branch points in $\Int(\D)$ (notice in particular that $p$ restricts to a cover $\partial \D \to \partial \D$).
	\item a simple $J$-holomorphic disk with corners and boundary on $L$, 
		\[ u' : (\D, \partial \D) \to (M, i(L) )\]	
\end{enumerate}
such that
\[ u = u' \circ p. \]
\end{coro}

Recall that there are two morphisms $\omega : \pi_2(M,L) \to \R$ and $\mu : \pi_2(M, L) \to \Z$ induced respectively by the symplectic area and the Maslov class. A Lagrangian submanifold $L \subset M$ is \emph{monotone} if there is a $\lambda >0$ such that
\[ \omega = \lambda \mu . \]
Denote by $N_L$ the minimal Maslov number of a Lagrangian submanifold $L$. Consider two transverse Lagrangian submanifolds $L_1$ and $L_2$ satisfying $N_{L_1} \geqslant 3$ and $N_{L_2} \geqslant 3$. As a direct application of Corollary \ref{coro:Jholdisksaremultiplycovered}, we will see that for a generic \emph{time-independent} $J \in \mathcal{J}(M,\omega)$, there is a well-defined Floer complex between these two objects. This differs from the usual situation where one usually considers time-dependent almost complex structures to achieve transversality (see \cite{Oh93}, \cite{FHS95}). 

\subsection{Outline of the proof of the Main Theorem}
We prove the Main Theorem \ref{Theo:Maintheorem} in several steps which follow Lazzarini's approach. We will emphasize along the argument the differences with \cite{La11}. 

First, we define a set $\mathcal{W}(u) \subset \D$ called the \emph{frame} of the disk which contains $\partial \D$. This is roughly the set of points where $u$ "overlaps" with its boundary. We then prove that this is actually a $\mathcal{C}^1$-embedded graph. We do this by providing an asymptotic development of the $J$-holomorphic curve around its corners.

The simple or multiply covered pieces are found by cutting the curve along the graph $\mathcal{W}(u)$. More precisely, we pick for each connected component $\Omega$ of $\D \backslash \mathcal{W}(u)$ a holomorphic embedding $h_\Omega : (\D,\partial \D) \to (\Omega, \mathcal{W}(u))$. The curve $u \circ h_\Omega$ satisfies $\mathcal{W}(u \circ h_\Omega) = \partial \D$ and is therefore either simple or multiply covered. The pieces of the decomposition are the simple curves underlying $u \circ h_\Omega$ for $\Omega$ a connected component.

Notice that a connected component $\Omega$ of $\D \backslash \mathcal{W}(u)$ is not necessarily simply connected, so we cannot immediately conclude that $u \circ h_\Omega$ factors through a simple disk. It turns out that if such a component exists, there is a simple holomorphic sphere $v : \C P^1 \to M$ such that $u(\D) = v(\C P^1)$. From this, we conclude that each piece is a disk. Here the details do not differ much from Lazzarini's paper (\cite{La11}).

\subsection{Outline of the paper}
\label{subsection:Outlinepaper}
The first section of the paper explains how to adapt Lazzarini's proof (\cite{La00}, \cite{La11}) to finite-energy curves with boundary on a given Lagrangian immersion $i : L \looparrowright M$. The frame along which the curve is cut into multiply covered pieces is a graph. This is the main technical part of the argument.  Second, we explain how to get the decomposition from this. 

The second section of the paper gives the proof of Corollary \ref{coro:Jholdisksaremultiplycovered}. In a second subsection, we will explain why this implies that the Floer complex is well-defined for a generic time-independent almost complex structure. 

Lastly, we give some expected applications of the main theorem to a count of holomorphic curves with boundary on the surgery of two Lagrangian embeddings. This fits in a more general program of Biran-Cornea and is the subject of work in progress.  

\subsection{Acknowledgements:}
This work is part of the author's doctoral thesis at the University of Montreal under the direction of Octav Cornea. I thank him for his thoughtful advice. I also thank Egor Shelukhin for helpful dicussions, as well as Emily Campling and Dominique Rathel-Fournier for help with the exposition.

\section{The Frame of a $J$-holomorphic curve}
\label{section:Frameofacurve}
Fix $u_1 : (S_1, \partial S_1) \to (M, i(L) )$ and $u_2 : (S_2, \partial S_2) \to (M, i(L))$ two finite-energy $J$-holomorphic curves with corners and boundaries on $L$.

We define the set of "bad points" of $u_1$ with respect to $u_2$ : 
	\[ \mathcal{C} \left(u_1, u_2 \right ) := u_1^{-1} \left (\ens{z \in \Int(S_1)}{du_1 (z) = 0} \right ) \cup u_1^{-1} \left (\ens{z \in \Int(S_2)}{du_2(z) = 0} \right ) \cup u_1^{-1}(i(R)) . \]
The following definition is due to Lazzarini (\cite{La11}).

\begin{defi}
\label{defi:coincidencerelation}
Suppose $z_1 \in \Int(S_1) \backslash \mathcal{C}(u_1, u_2)$ and $z_2 \in \Int(S_2) \backslash \mathcal{C}(u_2, u_1)$. We say that $z_1 \mathcal{R}_{u_1}^{u_2} z_2 $ if and only if for any open neighborhoods $V_1 \ni z_1$ (resp. $V_2 \ni z_2$), there are open neighborhoods $\Omega_1 \ni z_1$ (resp. $\Omega_2 \ni z_2$) in $V_1$ (resp. $V_2$) such that
	\[ u_1(\Omega_1) = u_2(\Omega_2) . \]
	
Now if $z_1 \in S_1$ and $z_2 \in S_2$, we say that $z_1 \mathcal{R}_{u_1}^{u_2} z_2$ if and only if there are sequences $(z_1^\nu)_{\nu \geqslant 0 }$ (resp. $(z_2^\nu)_{\nu \geqslant 0}$) such that $z_1^\nu \to z_1$ (resp. $z_2^\nu \to z_2$) and 
	\[ \forall \nu \geqslant 0,  z_1^\nu \mathcal{R}_{u_1}^{u_2} z_2^\nu . \] 
	 	
\end{defi}

We now define the graph along which we will cut to get the simple pieces of the curve.

\begin{defi}
\label{defi:frameofanholomorphiccurve}
The \emph{frame} of $u_1$ with respect to $u_2$ is the set of points related to the boundary of $S_2$:
	\[ \mathcal{W}(u_1, u_2) := \mathcal{R}_{u_1}^{u_2} \left ( \partial S_2 \right ) .\]
The \emph{completed frame} of $u_1$ with respect to $u_2$ is the union of this with $\partial S_1$:
	\[ \overline{\mathcal{W}}(u_1, u_2) := \mathcal{R}_{u_1}^{u_2} \left ( \partial S_2 \right )\cup \partial S_1 .\]
\end{defi}

\begin{rk}
	If $u$ is a $J$-holomorphic curve then $\partial S \subset \mathcal{W}(u,u)$, so
	\[\overline{\mathcal{W}}(u,u) = \mathcal{W}(u,u). \]
	From now on, we will abbreviate $\mathcal{W}(u) := \mathcal{W}(u,u)$. 
\end{rk}

In this section, we shall prove that the completed frame $\overline{\mathcal{W}}(u_1, u_2)$ is a $\mathcal{C}^1$ embedded graph in $S_1$. This is however not the case for $\mathcal{W}(u_1, u_2)$. Along the way, we will prove important properties of the relation $\mathcal{R}_{u_1}^{u_2}$, always following Lazzarini's proof. 

\subsubsection{Examples of frames and the decomposition}
As explained in the introduction, the simple pieces of the curve are found among the connected components of $\D \backslash \mathcal{W}(u) $.

The decomposition may introduce corner points which do not appear in the original curve. This is shown in the example below.

\begin{figure}
\captionsetup{justification=centering,margin=2cm}

\begin{tikzpicture}[scale = 0.75]
\draw (10,0) rectangle (16,6);
\draw[blue] (10,3) -- (16,3);
\draw[blue] (10,1.5) to [out=0,in=250] (12,3) to [out=70, in = 180] (13.5, 5.5) to [out= 0, in = 90 ](14.5,4.5) to [out= 270, in = 0](13.5,3.7) to [out = 180, in = 270](13.1,4.14) to [out = 90, in = 180] (13.34, 4.4) to [out = 0, in = 270] (13.94, 4.68) to [out = 90, in = 0] (13.7, 4.88) to [out = 180, in = 90] (13.54, 4.6) to [out = 90, in = 90] (13.54, 3.60) to [out = 270, in = 110](13.62, 3) to [out = 290, in = 180](16,1.5);

\draw[blue,fill = gray!20] (12,3) to [out=70, in = 180] (13.5, 5.5) to [out= 0, in = 90 ](14.5,4.5) to [out= 270, in = 0](13.5,3.7) to [out = 180, in = 270](13.1,4.14) to [out = 90, in = 180] (13.34, 4.4) to [out = 0, in = 270] (13.94, 4.68) to [out = 90, in = 0] (13.7, 4.88) to [out = 180, in = 90] (13.54, 4.6) to [out = 90, in = 90] (13.54, 3.60) to [out = 270, in = 110](13.62, 3) to (12,3) ;

\draw (12,3) node[below right]{\tiny $x_1$};
\draw (13.55, 4.4) node[above left]{ \tiny $x_3$};
\draw (13.55, 3.7) node[below right]{ \tiny $x_2$};
\draw (13.62, 3.0) node[below left]{\tiny $x_4$};

\draw (3,3) circle (3);
\draw (0,3) node {$\bullet$};
\draw (6,3) node {$\bullet$};
\draw (1, 5.24) node {$\bullet$};
\draw (2, 5.83) node {$\bullet$};
\draw (5, 5.24) node {$\bullet$};
\draw (4, 5.83) node {$\bullet$};

\draw (1,5.24) to [out=-10, in = 255] (2, 5.83);
\draw (5,5.24) to [out = 180, in = 285] (4, 5.83);

\draw[->] (7,3) to [out = 45, in = 180](8,3.5) to [out = 0, in = 135](9,3);
\draw (8, 3.5) node[above] {$u$}; 
\end{tikzpicture}

\caption{The immersion $i$ (blue, on the right), the disk $u$ (shaded, on the right) and its frame $\mathcal{W}(u)$ on the left}

\label{Figure : immersion dans le tore}

\end{figure}
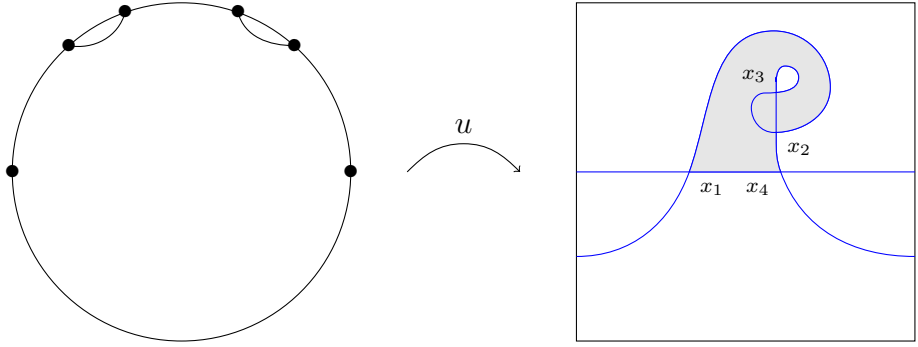

\begin{Ex}
Consider the 2-dimensional torus $\T^2 : = \R^2 / \Z^2$ equipped with the standard area form $d x \wedge dy$ and the standard complex structure.   

We let $i$ be the immersion of two copies of $S^1$ drawn in figure \ref{Figure : immersion dans le tore}. Moreover, we let $u$ be a $J$-holomorphic polygon with corners and boundary on $L$ whose image is represented in figure \ref{Figure : immersion dans le tore}. The parameterization of $u$ is chosen so that $u(-1) = x_1$ and $u(1) = x_4$. These are the only corner points of $u$.

The reader may check that the frame of $u$ is a graph with four vertices which map to the double points of $i$. The restriction of $u$ to each connected component of $\D \backslash \mathcal{W}(u)$ is a simple $J$-holomorphic curve with corners. Notice that each piece now has corners which map to $x_2$ and $x_3$. These corners did not appear in $u$.    	
\end{Ex}

Moreover, the frame need not be connected, as shown by the following example.

\begin{Ex}
Consider $\C P^1 = \C \cup \{ \infty \}$ equipped with its standard complex structure and let $L \subset \C$ be the ellipse with semi-major axis $\frac{5}{2}$ and semi-minor axis $\frac{3}{2}$. We consider a map with domain the disk of radius $2$
\[ u : \begin{array}{ccc} 
 	\D(0,2) & \to & \C \cup \{ \infty \} \\
 	z & \mapsto & z + \frac{1}{z}  	
 \end{array}. \]
We claim that the frame of $u$ is given by $\partial \D(0,2) \cup \partial \D\left (0,\frac{1}{2} \right )$. Notice first that \[ \mathcal{W}(u) \subset u^{-1}(L) = \partial \D(0,2) \cup \partial \D \left (0,\frac{1}{2} \right ) .\]
To prove the other inclusion, let $z  \in \partial \D \left (0,\frac{1}{2} \right)$ and $(\varepsilon_\nu)_{\nu \in \N}$ be a sequence of positive real numbers converging to $0$. Put $z_\nu = (1+\varepsilon_\nu)z$. Then the sequences $(z_\nu)$ and $\left ( \frac{1}{z_\nu} \right)$ satisfy $z_\nu \mathcal{R}_u^u \frac{1}{z_\nu} $ since $u(z_\nu) = u \left ( \frac{1}{z_\nu} \right)$ and $u'(z_\nu), u' \left (\frac{1}{z_\nu} \right)$ are non-zero. Hence $z \mathcal{R}_u^u \frac{1}{z} \in \partial \D$, so $z \in \mathcal{W}(u)$
\end{Ex}

\subsection{Local coordinates around double points of the immersion}
\label{subsection:localcoordinates}

\subsubsection{Some linear symplectic geometry}
\label{subsubsection:linearsymplectic}	  
Let $(V^{2n}, \omega)$ be a symplectic vector space of complex dimension $n$ and $J \in \mathcal{J}(V,\omega)$ be a compatible linear complex structure. Moreover, let $L_1$ and $L_2$ be transverse Lagrangian subspaces.
	  
Then there are reals $ 0 < \alpha_1 \leqslant \ldots \leqslant \alpha_n < \pi$ and a linear symplectic map $f : \left (V, \omega \right ) \to \left (\C^n, \omega_{\text{std}} \right )$ such that    
	  \[ f(L_1) = \R^n , f(L_2) = e^{i \alpha_1} \cdot \R \times \ldots e^{i\alpha_n} \cdot \R , f^* i = J .\]
The real numbers $\alpha_1, \ldots,  \alpha_n \in (0, \pi)$ do not depend on the choice of $f$ and are called the \emph{Kähler angles} of the pair $(L_1, L_2)$.

Moreover, if $\alpha \in (0, \pi)$, we define the vector subspace
	\[ V_\alpha := \ens{v \in \R^n}{e^{i \alpha} v \in e^{i \alpha_1} \R^n \times \ldots e^{i\alpha_n} \R^n} . \]
Notice that we have a direct sum decomposition
	\[ \R^n = \bigoplus_{\alpha \in \{\alpha_1, \ldots, \alpha_n\}} V_\alpha .\]
We let $\pi_\alpha : \R^n \to V_\alpha$ be the linear projection on $V_\alpha$ with respect to this decomposition. We will abuse notation slightly and call its complexification $\pi_\alpha : \C^n \to \C^n$ as well.
	
\subsubsection{A bit of vocabulary}
Let $(p,q) \in R$, since $d i_p$ (resp. $d i_q $) is an immersion, there is an open neighborhood $U_p \ni p$ (resp. $U_q \ni q$) such that $i_{\vert U_p}$ (resp. $ i_{\vert U_q} $) is an embedding. We call the submanifold $i(U_p)$ (resp. $i(U_q)$) the \emph{branch of $i$ at $p$} (resp. at $q$) and denote it by $L_p$ (resp. $L_q$).

In what follows, we will sometimes forget about $U_p$ and denote by $L_p$ the image of any neighborhood of $p$ on which $i$ is an embedding. 
	  
\subsubsection{Some local charts}
\label{subsubsection:localcharts}
We can now state

\begin{prop}
\label{prop:firstlocalchart}
	Let $(p,q) \in R$ and denote $x = i(p) = i(q)$.
	
	Then there are open neighborhoods $U$ of $0$ in $\C^n$, $V$ of $x$ in $M$, $U_p$ (resp. $U_q$) of $p$ (resp. $q$) in $L$ together with a smooth chart $\phi : U \to V$ satisfying the following properties.
	\begin{enumerate}[label=(\roman*)]
		\item Let $g_J : = \omega(\cdot, J \cdot)$ be the metric induced by the almost complex structure $J$, and $g_{\mathrm{std}}$ be the standard scalar product on $\C^n$. We have
		\[  \phi^* g_J(0) = g_{\mathrm{std}} ,  \phi^* J (0) = i . \]
		\item The chart maps the branches of $i$ at $x$ to linear subspaces :
		\[ \phi(U \cap \R^n) = i(U_p), \ \phi(U \cap e^{i \alpha_1} \cdot \R \times \ldots \times e^{i\alpha_n} \cdot \R) = i(U_q) . \]
		Here $\alpha_1 \leqslant \ldots \leqslant \alpha_n \in (0, \pi)$ are the Kähler angles of the pair $(T_x L_p , T_x L_q)$ with respect to the complex structure $J$.  	
	\end{enumerate}
\end{prop}	  

\begin{proof}
	There is a smooth chart $\tilde{\phi} : U \subset \C^n \to V \subset M$ such that $\tilde{\phi}(\R^n \cap U) = L_p \cap V$ and $\tilde{\phi} (i \cdot \R^n \cap U) = L_q \cap V$.
	
	We now modify $\tilde{\phi}$ so that it satisfies the assertions of the proposition. For this pick an orthonormal basis (with respect to the metric $g_J$) $\mathcal{B} =(e_1, \ldots, e_n)$ of $T_x L_p$. We assume that, with respect to the complex coordinates given by $\mathcal{B}$, we have 
	\[T_x L_q = e^{i\alpha_1} \cdot \R \times \ldots \times e^{i\alpha_n} \cdot \R . \]
	
	We put $(f_1, \ldots, f_n) := (d \tilde{\phi}^{-1}(e_1), \ldots, d \tilde{\phi}^{-1}(e_n) )$. Pick a real linear isomorphism $A : \C^n \to \C^n $ such that the image of the canonical basis of $\R^{2n}$ is the basis $(f_1, \ldots, f_n)$. The sought-after local chart is $\tilde{\phi} \circ A$ (it is defined on a small enough ball).  
\end{proof}

We will modify this chart to get a more precise behavior along $\R^n$.

\begin{prop}
\label{prop:secondlocalchart}
	There is a smooth local chart $\phi : U \subset \C^n \to V \subset M$ such that
	\begin{enumerate}[label=(\roman*)]
		\item we have $\phi^*J_{\vert \R^n} = i$ and $\left (\phi^* g_J \right )_0 = g_{\mathrm{std}}$,
		\item the preimages of the branches at $x$ are linear subspaces
		\[ \phi^{-1}(L_p \cap V) = \R^n \cap U, \phi^{-1}(L_q \cap V) = \left ( e^{i \alpha_1} \cdot \R \times \ldots \times e^{i \alpha_n} \cdot \R \right ) \cap U . \]
	\end{enumerate}
\end{prop}

\begin{proof}
		By the preceding Proposition \ref{prop:firstlocalchart}, one can assume that the two branches of the immersion are the given linear Lagrangian subspaces, that the almost complex structure $J$ satisfies $J(0) = \R^n$ and that the metric $g_J$ satisfies $g_J(0) = g_{\text{std}}$.
		
		Now choose $\psi : W \to U \subset \C^n$ a local chart such that $\psi^* J_{\vert \R^n} = i $ and $d \psi (0) = \Id$ (such a chart always exists, see the construction in \cite[lemma 3.7]{La11}). 
		
		Notice that $\psi^{-1}(L_q)$ is an embedded submanifold whose tangent space at $0$ is transverse to $\R^n$ (it is given by $e^{i \alpha_1} \cdot \R \times \ldots \times e^{i \alpha_n} \cdot \R$ ). Therefore the implicit function theorem implies that there is a smooth map $f : W \cap \R^n \to \R^n$ such that 
		\[ \psi^{-1}(L_q) = \ens{f(y)+ i y}{y \in \R^n}. \]
		
		Consider the map $\phi(x + iy) = f(y) - d f(0) \cdot y + x + iy $. Its differential is given by the matrix 
		\[ \begin{pmatrix} \Id & d f_y - d f_0 \\ 0 & \Id \end{pmatrix} ,\]
		so $d \phi_0 = \Id$ and $\phi$ is a local diffeomorphism.
		
		A small computation shows that $df_0$ is given by a diagonal matrix
		\[ \begin{pmatrix} \cot{\alpha_1} & & \\ & \ddots & \\ & & \cot{\alpha_n} \end{pmatrix}.\]
		
		Hence for $x = (x_1, \ldots , x_n)\in \R^n$, we have $\phi(x) = (f(0) + 0 +x , 0) \in \R^n$. Moreover, notice that 
		\[ df_0(x_1 \sin \alpha_1, \ldots, x_n \sin \alpha_n ) = (x_1 \cos \alpha_1, \ldots, x_n \cos \alpha_n), \]
		hence
		\[ \phi \left (x_1 e^{i\alpha_1}, \ldots, x_n e^{i \alpha_n} \right ) = f \left (x_1 \sin \alpha_1, \ldots, x_n \sin \alpha_n \right ) + i \left (x_1 \sin \alpha_1, \ldots , x_n \sin \alpha_n \right), \]
		
		so $\phi(\R^n) = e^{i\alpha_1} \cdot \R \times \ldots \times e^{i \alpha_n} \cdot \R$. 
		Moreover, for $x \in \R^n$, we have $d \phi_x = \Id$ so $\phi^*J_x = i$ and $\phi^*g_J (0) = g_{\text{std}}$.				
\end{proof}

\subsubsection{Behavior of a $J$-holomorphic curve around an interior point}

The asymptotic behavior of a $J$-holomorphic curve around an interior point can be described quite precisely. For instance, the following is proved in Lazzarini's paper \cite[Proposition 3.3]{La11}.

\begin{prop}
\label{prop: local behavior of a curve around an interior point}
Assume that $J : \D \to GL(2n,\R)$ is a $\mathcal{C}^1$ map such that $J^2 = - \Id$	and $J(0) = J_{\text{std}}$ is the standard complex structure. Let $u : S \to \C^n$ be a $J$-holomorphic curve with $u(0) = 0$. Then there are
\begin{enumerate}
	\item an integer $k \geqslant 1$,
	\item a $\mathcal{C}^1$-local chart $\phi : \Omega \to \D $ with $\Omega$ and open neighborhood of $0$ in $\D$ and $\phi(0) = 0$,
	\item a positive real number $\lambda_u > 0$ and a matrix $A \in U(n)$,
\end{enumerate}
such that 
\[ u \circ \phi(z) = \lambda_u A \left (z^k, U(z) \right ), \]
with $U(z) = O \left (z^{k+1} \right )$.
\end{prop}

\subsubsection{Behavior of a $J$-holomorphic curve around a double point}
\label{subsubsection:chartatacornerpoint}
In this subsection, we describe, along the lines of \cite[section 3.2]{La11}, the local form of a curve around the corner points.
 
For this, let us fix $(p,q) \in R$ and put $x = i(p) = i(q)$. As usual, we call \[ 0< \alpha_1 \leqslant \ldots \leqslant \alpha_n < \pi\] the Kähler angles of the pair $(T_x L_p, T_x L_q)$ with respect to $J$.
We also choose a chart $\phi : U \to V$ such as the one given in Proposition \ref{prop:firstlocalchart}.  

\begin{prop}
\label{prop:asymptoticdevelopmentaroundacornerpoint}
	Let $\D^+ = \ens{x+iy}{\vert x+iy \vert < 1,  y \geqslant 0}$ be the unit upper half-disk and $\D^+_\R = \D^+ \cap \R$ be its real part.
	 
	Let $u : (\D^+,\D^+_\R) \to (M, i(L))$ be a non-constant $J$-holomorphic half-disk with boundary on $L$ and finite energy (i.e. $\int u^* \omega < + \infty$). Assume that the lift $\gamma_{[0,1)}$ (resp. $\gamma_{(-1,0]}$) of $u_{\vert [0,1)}$ (resp. $u_{\vert (-1,0]}$) to $L$ satisfies $\gamma_{[0,1)}(0) = p$ (resp. $\gamma_{[0,1)}(0) = q$ )\footnote{Geometrically this means that the curve has right boundary condition along the branch $L_p$ and left boundary condition along the branch $L_q$.}.
	
	Then there are integers $k \in \{ 1, \ldots, n\}$ and $m \geqslant 0$, together with a positive real number $\delta > 0$ and a vector $a_k \in V_{\alpha_k}$\footnote{see \ref{subsubsection:linearsymplectic} for the the definition of $V_\alpha$} such that
	\[ \phi^{-1} \circ u (z) = a_k z^{\frac{\alpha_k}{\pi} + m} + o \left (z^{\frac{\alpha_k}{\pi} + m + \delta} \right). \]
	Moreover, we have
	\[ d (\phi^{-1} \circ u) (z) = \left ( \frac{\alpha_k}{\pi}+ m \right ) a_k z^{\frac{\alpha_k}{\pi} + m-1} + o \left (z^{\frac{\alpha_k}{\pi} + m + \delta -1 } \right ) .\]
	Note that this implies that there are no critical points in a sufficiently small punctured neighborhood of a corner point.
\end{prop}

\begin{rk}
\label{rk:definition of the multiplicity}
We call the integer $m+1$ the \emph{multiplicity} of the curve $u$ at $0$.  	
\end{rk}

\begin{proof}
	This is an application of a theorem of Robbin and Salamon (\cite[Theorem B]{RS01}) on the asymptotics of a finite-energy $J$-holomorphic strip. 
	
	To see this, fix $r>0$, and define the strip-like end 
	\[\varepsilon_r : \begin{array}{ccc} S := [0,+\infty) \times [0,1] & \to & (\D^+,\D^+_\R) \\ (s,t) & \mapsto & -re^{-\pi(s+it)} \end{array}.\] 
	For $r \ll 1$ consider the map 
	\[ \tilde{u} := \phi^{-1} \circ u \circ \varepsilon_r : S \to \C^n .\]
	
	Then $\tilde{u}$ is pseudo-holomorphic with respect to the almost complex structure $\phi^* J$, has finite energy with respect to the metric $g_{\phi^*J}$ and satisfies the boundary condition 
	\[ \tilde{u} \left ( [0,+\infty) \times \{0\} \right ) \subset  e^{i \alpha_1} \cdot \R \times \ldots \times e^{i \alpha_n} \cdot \R,\ \tilde{u}([0,+\infty) \times \{1\}) \subset \R^n . \] 
	Moreover, since $u$ is continuous, for $r$ small enough the map $\phi^{-1} \circ u \circ \varepsilon_r$ has relatively compact image in $\C^n$. Hence by \cite[Theorem A]{RS01}, $\tilde{u}(s,\cdot)$ converges uniformly to $x$ as $s \to + \infty$ and its derivative $\partial_s u$ decays exponentially with respect to the usual $\mathcal{C}^{\infty}$ pseudo-distance.
	
	We can now apply \cite[Theorem B]{RS01}. There exist a $\lambda > 0$ and a map $v : [0,1] \to \C^n$ such that 
	\[ i \partial_t v = \lambda v ,\ v(0) \in e^{i \alpha_1} \cdot \R \times \ldots \times e^{i \alpha_n} \cdot \R,\ v(1) \in \R^n ,  \]
	and a $\delta > 0 $ such that 
	\[ u(s,t) = \exp_0 \left (- \frac{1}{\lambda} e^{-\lambda s} v(t) + w(s,t) \right ), \norm{w}_{\mathcal{C}^k} \leqslant c_k e^{-(\lambda + \delta)s} .\]
	A small computation shows that there exist an integer $m \geqslant 0$, an $\alpha_k$ and a vector $v_k \in V_{\alpha_k}$ such that 
	\[ \lambda = \alpha_k + m \pi, \ v(t) = e^{i\alpha_k} e^{-i(\alpha_k + m \pi)t} .\] 
	Now notice that if $z = -e^{-\pi(s+it)}$ , then $z^{\frac{\alpha_k}{\pi} + m } = e^{-  s \left (\alpha_k + m \pi  \right )}  e^{i(\alpha_k + m \pi)(1-t)}$. Hence,
	\[ u(z) = \exp_0 \left (-\frac{(-1)^m}{\lambda} e^{-(\alpha_k + m \pi)s} e^{i(\alpha_k + m \pi)(1-t)} \right ),  \]
	and so
	\[ u(z) = \exp_0 \left (-\frac{(-1)^m}{\lambda}  z^{\alpha_k + m \pi} + w(z) \right ). \]
	This gives the relevant estimate.
	
	The estimate on the derivative follows easily from the chain rule applied to $\phi \circ u \circ \varepsilon$.
\end{proof}

From now on, we will work locally in $M$ with the help of the chart given by Proposition \ref{prop:secondlocalchart}. Therefore we shall consider $J$-holomorphic curves with values in $\C^n$ equipped with an almost complex structure $J$ such that $J_{\lvert \R^n} = J_{\text{std}}$. We assume these curves have boundaries on the union of the branch $ L_p = \R^n$ and $L_q = e^{i \alpha_1} \cdot \R \times \ldots \times e^{i\alpha_n} \cdot \R$. We shall describe their behavior around the double point $0$.

\begin{prop}
\label{prop:localformaroundacornerpoint}
	Assume that $u : (\D^+, \D^+_\R) \to (\C^n, L_p \cup L_q)$ satisfies the hypothesis of Proposition \ref{prop:asymptoticdevelopmentaroundacornerpoint}. 
	
	Then there exist 
	
	\begin{enumerate}
		
	\item an open neighborhood $\Omega$ of $0$ in $\D^+$,
	\item a $\mathcal{C}^1$ chart 
	\[ \psi : (\Omega, \Omega \cap \R) \to (\D^+, \D^+_\R), \] 
	\item a linear isometry $A_u \in \mathcal{L}(\R^{\dim V_{\alpha_k}}, V_{\alpha_k} )$ and a $\lambda_u \in \R^+$ such that
	\[ \pi_{\alpha_k} \left (u \circ \psi(z) \right ) = \lambda_u A_u \left (z^{\frac{\alpha_k}{\pi} + m },\tilde{U}(z) \right )  \]
	with $\tilde{U}(z) = o \left (z^{\frac{\alpha_k}{\pi} + m + \delta} \right )$.
	\end{enumerate}
	Moreover, if 
	\[ U(z) = \sum_{\alpha \in \{\alpha_1, \ldots , \alpha_n\} \backslash \{ \alpha_k\}} \pi_\alpha \left (u \circ \psi (z) \right ), \]
	we have
	\[ U(z) = o\left (z^{\frac{\alpha_k}{\pi} + m + \delta} \right ), \ d U(z) = o\left (z^{\frac{\alpha_k}{\pi} + m + \delta - 1} \right ) . \]
	
\end{prop}

\begin{proof}
	Replacing $u$ by $\phi \circ u$ we can assume that $u$ has values in $\C^n$. Using Proposition \ref{prop:asymptoticdevelopmentaroundacornerpoint}, there are $k$ and $a_k \in V_{\alpha_k}$ such that 
	\[ u(z) = a_k z^{\frac{\alpha_k}{\pi} + m} +o \left (z^{\frac{\alpha_k}{\pi} + m + \delta} \right ). \]
	Choose an isometry $A \in \mathcal{L}(\R^n, V_{\alpha_k} )$ and a $\lambda_u > 0$ such that $A(1,0) = \lambda_u a_k$, then we have
	 \[ \pi_{\alpha_k} ( u(z) ) = \lambda_u z^{\frac{\alpha_k}{\pi} + m}A(1 + a(z), U_1(z)) \] 
	 with $a(z) \in \C$, $a(z) = o(z^\delta)$ and $U_1(z) = o(z^{\delta})$.
	 
	 Now if $r > 0$ is small enough, define $\phi$ on $\D(0,r)$ by 
	\[ \phi(z) = z \left ( 1 + a(z) \right )^{\frac{1}{\frac{\alpha_k}{\pi} + m}} . \]
	The map $\phi$ is $\mathcal{C}^1$ on $\D^+(0,r) \backslash \{ 0 \}$, and if $z \neq 0$ we have
	\[ \phi'(z) =  (1+a(z))^{\frac{1}{\frac{\alpha_k}{\pi} + m}} d z + \frac{z a'(z)}{\frac{\alpha_k}{\pi} + m}(1 + a(z))^{\frac{1}{\frac{\alpha_k}{\pi} + m}-1} .\]
	Therefore $\phi'(z) \to 1$ as $z \to 0$. Hence, $\phi$ extends to a $\mathcal{C}^1$ map on $\D^+(0,r)$.
	
	Now if $z \in \R_+$, we have $\pi_{\alpha_k}(u(z)) \in \R^{ \dim V_{\alpha_k}}$, so $\lambda_u z^{\frac{\alpha_k}{\pi}+m}(1 + a(z)) \in \R$ and $1+a(z) \in \R$. If $z \in \R_-$, since $\pi_{\alpha_k}(u(z)) \in e^{i\alpha_k} \cdot \R^{\dim V_{\alpha_k}}$ we similarly obtain $1+a(z) \in \R$.
	
	Since $a(z) \to 0$ as $z \to 0$, we can assume that for $z \in \D(0,r) \cap \R$ we have $1+a(z) \in \R^+$. Hence $\left ( 1 + a(z) \right )^{\frac{1}{\frac{\alpha_k}{\pi} + m}} \in \R$ and $\phi(z) \in \R$.  
	
	 We can now use the Schwarz reflection principle to see that $\phi$ extends to a map defined on $\D(0,r)$ with invertible differential at the origin. Therefore it admits a local inverse. We will now assume that $r>0$ is small enough so that $\phi$ is actually invertible.
	
	The image of $\D(0,r)$ by $\phi$ is an open subset of $\C$ with boundary a $\mathcal{C}^1$ simple closed curve. By the Jordan curve theorem, this image is cut by the real line $\R$ into two connected components. These are necessarily the images of the connected components of $\D(0,r) \backslash \R$ by $\phi$. We conclude that $\phi(\D^+(0,r))$ is a subset of $\HP$.
	
	Now 
	\[ \pi_{\alpha_k} ( u(z) ) = \lambda_u A \left( \phi(z)^{\frac{\alpha_k}{\pi}+m}, U_1(z) \right ),\]
	and so 
	\[ \pi_{\alpha_k} ( u(\phi^{-1}(z)) ) =  \lambda_u A \left( z^{\frac{\alpha_k}{\pi}+m}, U_1\circ \phi^{-1}(z) \right ) .\]
\end{proof}

Let us recall the analog of this (Proposition \ref{prop:localformaroundacornerpoint}) in the case of a curve with boundary along a single branch of the immersion. This is \cite[Lemma 3.5]{La11} and it is proved in the same manner as above (with a bit less trouble).

\begin{prop}
\label{prop:local form along a branch}
Assume that $u:(\D^+, \D^+_\R) \to (\C^n,L_p)$ is a finite-energy, $J$-holomorphic curve with $u(0) = 0$.

There are a matrix $A_u \in O_n(\R)$, a $\lambda_u >0$, a natural $m \in \N$ and $\psi$ a $\mathcal{C}^1$ local chart around $0$ such that
	\[ u \circ \phi (z) = \lambda_u A_u \left (z^m, U(z) \right), \]  
	with $U(z) = o \left (z^m \right )$ and $dU(z) = o \left (z^{m-1} \right )$.
\end{prop}

These two propositions allow us to give the local behavior of these curves when they have boundary conditions along $L_q$ rather than $L_p$.

For this let us introduce $D_{\alpha_1,\ldots , \alpha_n}$ the $n \times n$ diagonal matrix with successive entries $e^{i\alpha_1}, \ldots, e^{i \alpha_n}$.

\begin{prop}
\label{prop: Local form along the other branch}
Assume that $u:\left (\D^+, \D^+_\R \right ) \to \left (\C^n, L_q \right )$ is a finite-energy $J$-holomorphic curve with $u(0) = 0$.

There are a matrix $B_u \in O_n \left (\R \right )$, a $\lambda_u > 0 $ such that 
	\[ u(z) = \lambda_u D_{\alpha_1, \ldots , \alpha_n} B_u \left (z^m, U(z) \right ), \]
with $U(z) = o \left( z^m \right )$ and $dU(z) = o \left ( z^{m-1} \right )$. 	
\end{prop}

\begin{proof}
	This follows directly from Proposition \ref{prop:local form along a branch}. To see this, consider the curve $v = D_{-\alpha_1, \ldots, - \alpha_n} u $.
	Then $v$ satisfies the hypotheses of \ref{prop:local form along a branch} with the complex structure $D_{- \alpha_1, \ldots , - \alpha_n} J D_{\alpha_1, \ldots, \alpha_n}$. This immediately gives the conclusion.	
\end{proof}

The same trick allows us to give a local form around a corner point.

\begin{prop}
\label{prop: local form of a corner point along the other branch}
Assume that $u: \left ( \D^+, \D^+_\R \right ) \to \left (\R^n, L_p \cup L_q \right ) $ is a finite-energy $J$-holomorphic curve such that $u(0) =0$ and $u \left ( [0,1) \right ) \subset L_q$ and $u \left((-1,0] \right ) \subset L_p$ (this implies in particular that there is a corner point at $0$).	

Then there exist 
	
	\begin{enumerate}
		
	\item an open neighborhood $\Omega$ of $0$ in $\D^+$,  
	\item a $\mathcal{C}^1$ chart 
	\[ \psi : (\Omega, \Omega \cap \R) \to (\D^+, \D^+_\R), \] 
	\item an linear isometry $B_u \in \mathcal{L}(\R^{\dim{V_{\alpha_k}}}, V_{\alpha_k} )$ and a $\lambda_u \in \R^+$ such that
	\[ \pi_{\alpha_k} \left (u \circ \psi(z) \right ) = \lambda_u e^{i \alpha_k} B_u \left (z^{\frac{\pi - \alpha_k}{\pi} + m },\tilde{U}(z) \right )  \]
	with $\tilde{U}(z) = o \left (z^{\frac{\pi - \alpha_k}{\pi} + m + \delta} \right )$.
	\end{enumerate}
	Moreover, if 
	\[ U(z) = \sum_{\alpha \in \{\alpha_1, \ldots , \alpha_n\} \backslash \{ \alpha_k\}} \pi_\alpha \left (u \circ \psi (z) \right ) \]
	we have
	\[ U(z) = o\left (z^{\frac{\pi - \alpha_k}{\pi} + m + \delta} \right ), \ d U(z) = o\left (z^{\frac{\pi - \alpha_k}{\pi} + m + \delta - 1} \right ) . \]
\end{prop}

\begin{proof}
	As before, consider the curve $v = D_{-\alpha_1, \ldots, -\alpha_n} u$. It satisfies the boundary condition $v \left ( [0,1) \right ) \subset \R^n$ and $v \left ( (-1,0] \right ) \subset e^{-i \alpha_1} \cdot \R \times \ldots \times e^{-i \alpha_n} \cdot \R$. Notice that the Kähler angles of the second boundary condition are given by $\pi-\alpha_1, \ldots , \pi - \alpha_n$. 
	
	We now apply Proposition \ref{prop:localformaroundacornerpoint} to obtain an $\alpha_k \in (0,\pi)$, a $\lambda_u > 0$, a linear map $B_u$ and a local $\mathcal{C^1}$ diffeomorphism $\psi$ such that 
		\[ \pi_{\alpha_k} \left ( v( \psi(z)) \right ) = \lambda_u B_u \left ( z^{\frac{\pi - \alpha_k}{\pi} +m }, U(z) \right ) \]
		with $U(z) = o (^{\frac{\pi - \alpha_k}{\pi} +m })$ and $\delta >0$. 
	Notice that $\pi_{\alpha_k} (D_{-\alpha_1, \ldots, -\alpha_n} u)= e^{-i \alpha_k} \pi_{\alpha_k}(u)$, so 
	\[ \pi_{\alpha_k} \left (u \circ \psi(z) \right ) = \lambda_u e^{i \alpha_k} B_u \left (z^{\frac{\pi - \alpha_k}{\pi} + m },\tilde{U}(z) \right ). \]				
\end{proof}

\subsection{The relative frame of the curve is a graph}
\label{subsection:relativeframe=graph}

In this subsection, we will explain how to adapt the argument of \cite{La11} to show that given two finite-energy $J$-holomorphic curves with boundary on $L$, their relative frame is a $\mathcal{C}^1$-embedded graph. Recall that $u_1$ and $u_2$ are two finite-energy $J$-holomorphic curves with corners and boundary on $L$.

First let us state \cite[Lemma 3.10]{La11}. The proof adapts without difficulty to our context.

\begin{prop}
\label{prop:Theframerelationisopen}
	Suppose that $p_1 : S_1 \times S_2 \to S_1$ is the projection onto the first factor. Then if $V_1 \subset S_1$ and $V_2 \subset S_2$, the map $p_1 : (V_1 \times V_2) \cap \mathcal{R}_{u_1}^{u_2} \to V_1$ is open if
	\begin{enumerate}
		\item either $V_2 \subset \Int(S_2)$ is open,
		\item or $V_1$ is an open set such that $V_1 \cap \mathcal{W}(u_1, u_2) \subset \partial S_1$.	
	\end{enumerate}

\end{prop}

\begin{proof}
		Let $q_1 \mathcal{R}^{u_2}_{u_1} q_2$ and let $V_i \ni q_i$ be two open neighborhoods satisfying $V_2 \subset \Int(S_2)$ and $V_1 \cap \mathcal{W}(u_1, u_2) \subset \partial S_1$. Assume that the $V_i$ are open half-disks or open disks and that $V_1 \backslash \{ 0 \} \cap \mathcal{C}(u_1, u_2) = \emptyset$.
		
		Up to reparameterization by $z \to u_1(\lambda z)$ with $\lambda > 0 $ small enough, we can assume that if $(z_1, z_2) \in V_1 \times V_2$ is such that $u_1(z_1) = u_2(z_2)$, then $\norm{z_2} \leqslant \frac{1}{2}$. 
		
		First, assume that $z_1$ is a corner point and $z_2$ is not. There are constants such that
		\[ \norm{u_1 (\lambda z)} \leqslant C_1 \lambda^{\frac{\alpha_i}{\pi} + m } \norm{z}^{\frac{\alpha}{\pi} + m }, \  \norm{u_2(z)} \geqslant C_2 \norm{z}^{k} .\]
		So if $u_2(z_2) = u_1(\lambda z_1)$, we have 
		\[ C_2 \norm{z}^{k_2} \leqslant C_1 \norm{\lambda}^{\frac{\alpha_i}{\pi} + m } \norm{z_1}^{\frac{\alpha_i}{\pi} + m}. \]
		Hence 
		\[ \norm{z_2} \leqslant \left ( \frac{C_1}{C_2} \lambda^{\frac{\alpha_i}{\pi} + m } \right )^{\frac{1}{k_2}} . \]
		The right term goes to zero as $\lambda \to 0^+$. Therefore, the result is true for $\lambda$ small enough.
		
		Second, assume that $z_2$ is a corner point and $z_1$ isn't. Then there are constants such that 
		\[ \norm{u_1(\lambda z)} \leqslant C_1 \lambda^k \norm{z}^k , \ \norm{u_2(z)} \geqslant C_2 \norm{z}^{\frac{\alpha}{\pi} + m } .\]
		So if $u_2(z_2) = u_1(\lambda z_1)$, we have 
		 \[ C_2 \norm{z_2}^{\frac{\alpha}{\pi} + m } \leqslant \norm{\lambda}^k \norm{z_1}^k , \]
		hence
		\[ \norm{z_2} \leqslant \left( \frac{C_1}{C_2} \lambda^k \right)^{\frac{1}{\frac{\alpha}{\pi}+ m}}. \]
		The right term goes to zero as $\lambda \to 0^+$. Therefore the result is true for $\lambda > 0$ small enough.

		Last assume that both $z_1$ and $z_2$ are corner points. Then there are constants such that
		\[ \norm{u_1 (\lambda z)} \leqslant C_1 \lambda^{\frac{\alpha_1}{\pi} + m_1 } \norm{z}^{\frac{\alpha_1}{\pi} + m_1 }, \  \norm{u_2(z)} \geqslant C_2 \norm{z}^{\frac{\alpha_2}{\pi} + m_2 } .\] 
		So if $u_2(z_2) = u_1(\lambda z_1)$, we have 
		 \[ C_2 \norm{z_2}^{\frac{\alpha_2}{\pi} + m_2 } \leqslant \norm{\lambda}^{\frac{\alpha_1}{\pi} + m_1} \norm{z_1}^{\frac{\alpha_1}{\pi} + m_1 }. \]
		Hence
		\[ \norm{z_2} \leqslant \left( \frac{C_1}{C_2} \lambda^{\frac{\alpha_1}{\pi} + m_1 } \right)^{\frac{1}{\frac{\alpha_2}{\pi}+ m_2}}. \]
		The right term goes to zero as $\lambda \to 0^+$. Therefore the result is true for $\lambda > 0$ small enough. 
				
		Now, let $\Omega = \mathcal{R}^{u_2}_{u_1} (V_2) \cap \left (\Int(V_1) \backslash \{ 0 \} \right ) \subset \Int(V_1 \backslash \{ 0 \}) \cup \partial S_1 $.
		\begin{itemize}
			\item We have that $\Omega \neq \emptyset$. 	Indeed, there are sequences $(q_{1,\nu})$ and $(q_{2, \nu})$ with values in $\Int(S_1) \backslash \mathcal{C}(u_1,u_2) $ and $\Int(S_2) \backslash \mathcal{C}(u_2,u_1)$ such that $q_{1,\nu} \to q_1$, $q_{2,\nu} \to q_2$ and $q_{1, \nu} \neq q_1$. Now for $\nu$ large enough, $q_{1, \nu} \in V_1 \backslash \{ q_1 \}$ and $q_{2, \nu} \in V_2$.
			\item The set $\Omega$ is open in $\Int(V_1)$ : if $z_1 \in \Omega$, then $z_1 \notin \mathcal{C}(u_1,u_2)$. Let $z_2 \in V_2$ be such that $z_1 \mathcal{R}^{u_2}_{u_1} z_2$. Then $z_2 \in \Int(V_2)$ since if $z_2 \in \partial S_2$ we have $z_1 \in \partial S_2$ which is a contradiction. Moreover, $d u_1 (z_1) \neq 0$ and $d u_2 (z_2) \neq 0$. So the restrictions of the two curves to small enough open neighborhoods of $z_1$ and $z_2$ are reparameterizations of each other. 
			\item The set $\Omega$ is closed in $\Int(V_1) \backslash \{0\}$ since if $z_{1,\nu} \to z \in \Int(V_1) \backslash \{ 0 \}$, there is $z_{2,\nu} \in V_2$ such that $z_{1,\nu} \mathcal{R}^{u_2}_{u_1} z_{2, \nu}$. One can assume that the sequence $(z_{2, \nu})$ converges to $z_2$. Since $\norm{z_2} \leqslant \frac{1}{2}$, we get $z_2 \in V_2 $.
		\end{itemize}
	Hence $\Omega = \Int(V_1) \backslash \{ 0 \}$ and the result follows by taking the closure of this in $V_1$ and $V_2$ since $\mathcal{R}_{u_1}^{u_2}$ is closed.				
\end{proof}

Let us also recall a characterization of simple curves with corners and boundary on $L$.

\begin{prop}
\label{prop:Characterizationofsimplecurves}
	Let $u : (S, \partial S) \to (M, L)$ be a finite-energy $J$-holomorphic curve with boundary in $L$. The curve $u$ is simple if and only if $\mathcal{R}^u_u$ is the trivial relation.	
\end{prop}

\begin{proof}
	If $\mathcal{R}^u_u$	 is non-trivial, it is easy to show that $u$ is not simple : see \cite[Corollary 3.16]{La11}.
	
	Assume that $\mathcal{R}^u_u = \Delta$ and let $\mathcal{N} = \ens{z \in \Int(S) \backslash \mathcal{C}(u,u)}{\# u^{-1}(u(z)) \geqslant 2}$. Suppose that $z_{1,\nu} \to z_1 \in \mathcal{N}$ and $u(z_{1,\nu}) = u(z_{2,\nu})$ with $z_{2,\nu} \to z_2 \in S$ and $z_{1, \nu} \neq z_{2, \nu}$. Then since $z_1 \notin \mathcal{C}_{u,u}$, we have that $u_1(z_1) \notin i(R)$, hence by \cite{La11}, $z_1 \mathcal{R}_u^u z_2$ and so $z_1 = z_2$. This is a contradiction since $d u (z_1) \neq 0$ and $u$ is locally injective around $z_1$. 
\end{proof}

We will now explain how to prove that $\mathcal{W}(u_1,u_2)$ is a $\mathcal{C}^1$ embedded graph. The proof is still an adaptation of \cite{La00} and \cite{La11} with special care given to corner points.

Let us fix $z_1 \in S_1$ and $z_2 \in \partial S_2$ such that $z_1 \mathcal{R}_{u_1}^{u_2} z_2$. We will show that the desired result holds locally around $z_1$.
 
 There are several cases to consider depending on the type of the points $z_1$ and $z_2$. The proofs of all of these follow a variation of the same scheme (and are therefore quite interchangeable). Namely
 	\begin{enumerate}
 		\item For $i = 1,2$ we find an expression of $u_i$ around $z_i$ of the type 
 		\[ u_i(z) = \lambda_i A_i(1,0) z^{c_i} + A_i(0, U_i(z)) \]
 		with $c_i$ a positive real number and $U_i(z) = o(z^{c_i})$ (for this we apply one of the Propositions \ref{prop: Local form along the other branch},\ref{prop:local form along a branch}, \ref{prop: local form of a corner point along the other branch}, \ref{prop:localformaroundacornerpoint} according to the type of $z_i$ ).
 		\item Since there are sequences $z_{1, \nu} \to z_2$ and $z_{2,\nu} \to z_2$	 which satisfy $u_1(z_{1,\nu}) = u_2(z_{2,\nu})$, we deduce that $A_1(1,0)$ and $A_2(1,0)$ are dependent over $\C$ and lie in the complexification of $V_{\alpha_p}$ for some $p$.
 		\item We then use the complexification of the standard scalar product to conclude that $\mathcal{W}(u_1,u_2)$ is included in a union of rays.
 	\end{enumerate}
 	
\begin{lemma}
\label{lemm: Cas 1}
	Assume that $z_1 \in \partial S_1$ and that $z_1$ and $z_2$ are not corner points.
	
	Moreover, we suppose that $u_1(z_1) = u_2(z_2) = i(p) = i(q)$ is a double point, that $u_1$ has boundary condition along the branch $L_p$ around $z_1$ and $u_2$ has boundary condition along the branch $L_q$ around $z_2$.
	
	Then there is an open neighborhood $\Omega$ of $z_1$ such that $\mathcal{W}(u_1,u_2) \cap \Omega$ is a $\mathcal{C}^1$-embedded graph in $\Omega$.
	
\end{lemma}

\begin{proof}
	Using Propositions \ref{prop:secondlocalchart}, \ref{prop:local form along a branch}, and \ref{prop: Local form along the other branch}, we can assume that
	\begin{enumerate}
		\item $u_1$ and $u_2$ have values in $\C^n$,
		\item $L_p$ is given by $\R^n$ and $L_q$ is given by $e^{i \alpha_1} \cdot \R \times \ldots \times e^{i\alpha_n} \cdot \R$,
		\item there are local $\mathcal{C}^1$ diffeomorphisms $\psi_1$ and $\psi_2$ around $z_1$ and $z_2$ respectively with images $\Omega_1$ and $\Omega _2$ such that
		\begin{eqnarray*}
			u_1 \left (\psi_1(z) \right ) & = & \lambda_1 A \left (z^k, U(z) \right ) \\
			u_2 \left ( \psi_2(z) \right ) & = & \lambda_2 D_{\alpha_1, \ldots, \alpha_n} B \left (z^m , \tilde{U}(z) \right). \\	
		\end{eqnarray*}	
	\end{enumerate}
Replacing $\Omega_1$ and $\Omega_2$ by smaller neighborhoods if necessary, we can assume that $\mathcal{C}(u_1,u_2) \cap \Omega_1 \subset \{ 0 \}$.

We claim that there is a complex $\mu \in \C \setminus \{ 0 \}$ such that $\mu A (1,0) = D_{\alpha_1, \ldots, \alpha_n} B (1,0)$ and that there is an $\alpha_k$ such that $A(1,0) \in V_{\alpha_k}$.

Since $z_1 \mathcal{R}_{u_1}^{u_2} z_2$, there are sequences $(z_{1, \nu})$ and $(z_{2, \nu})$ of points distinct from $z_1$ and $z_2$ such that $z_{1, \nu} \to z_1$ and $z_{2,\nu} \to z_2$ and $u_1(z_{1, \nu}) = u_2(z_{2, \nu})$. There are $\mu \in \C$ and $v \in \R^{n-1}$ such that $ D_{\alpha_1, \ldots, \alpha_n} B (1,0)= \mu A(1,0) + A(0,v)$.

If by contradiction $ \mu = 0$, from the equality
	\[ \lambda_1 z_{1, \nu}^k A(1,0) + \lambda_1 A(0, U(z_{1,\nu})) = \lambda_2 z_{2,\nu}^m A(0,v) + \lambda_2 D_{\alpha_1, \ldots, \alpha_n} B(0,\tilde{U}) ,\]
	we get $z_{1, \nu}^{k} = o (z_{2, \nu}^m)$. Therefore, we would have $u_1(z_{1,\nu}) = o(z_{2,\nu}^m) = o(u_2(z_{2,\nu^m}))$. This is of course a contradiction.
	
From this, we deduce $\lambda_1 z_{1, \nu}^k \sim \mu \lambda_2 z_{2, \nu}^m$. Denote by $\pi: \R^n \to \R^n$ the real orthogonal projection onto $A(1,0)^{\perp}$. Since $A$ is orthogonal, we get
\[ o(z_{1, \nu}^k) = \pi(u_1(z_{1,\nu})) = \pi(u_2(z_{2,\nu})) = z_{2, \nu}^m \pi(D_{\alpha_1, \ldots, \alpha_n} B (1,0) +o(z_{2,\nu}) ).\]
Hence, $\pi(D_{\alpha_1, \ldots, \alpha_n} B (1,0)) = 0$.

Moreover, we have $\mu A(1,0) = D_{\alpha_1, \ldots , \alpha_n}B(1,0) \in e^{i \alpha_1} \cdot \R \times \ldots \times e^{i \alpha_n} \cdot \R$ and $A(1,0) \in \R^n$, so $A(1,0) \in V_{\alpha_p}$ for some $p$. We conclude that $\mu$ has argument $\alpha_k$ mod $ \pi$.    
 
Assume that $z \in \Omega_1 \cap \mathcal{W}(u_1,u_2)$. Then $u(z) \in L_q$. Denoting by $\langle \cdot, \cdot \rangle$ the complexification of the usual scalar product on $\R^n$, we have
\[ \lambda_1 z^k = \langle u_1(z), A(1,0) \rangle. \]
Since $u_1(z) \in L_q$ and $A(1,0) \in V_{\alpha_k}$, we have $\langle u_1(z), A(1,0) \rangle \in e^{i\alpha_p} \cdot \R$, so $z^k \in e^{i \alpha_p} \cdot \R$.

We conclude that $\mathcal{W} (u_1,u_2) \subset A$ where $A$ is the union of rays given by  
\[ A : = \left (\bigcup_{q} e^{i \frac{\alpha_p}{k} + i \frac{2 \pi q}{p}} \cdot \R_+ \right ) \cup \left ( \bigcup_{q} e^{i \frac{\alpha_p}{k} + i  \frac{(2q+1) \pi}{p} } \cdot \R_+ \right ). \]

 We claim that the frame $\mathcal{W}(u_1,u_2) \backslash \{0\}$ is a (possibly empty) union of connected components of $A \backslash \{ 0 \}$. We prove this by showing that it is an open and closed subset of $A \backslash\{ 0 \} $.
 
 Notice that $\mathcal{W}(u_1,u_2) = \mathcal{R}(\partial S_2)$ is closed, since $\mathcal{R}_{u_1}^{u_2}$ and $\partial S_2$ are both closed. We conclude that $\mathcal{W}(u_1,u_2) \backslash \{0\}$ is closed in $A \backslash \{ 0 \}$. 
 
 Since $\Omega_2 \cap \mathcal{C}(u_1,u_2) \subset \{0\}$ any point of $\mathcal{W}(u_1,u_2) \backslash \{ 0 \}$ is not in $\mathcal{C}(u_1,u_2)$. Therefore, we can apply the proof of \cite[Theorem 3.18]{La11} to conclude that $\mathcal{W}(u_1,u_2)\backslash \{ 0 \}$ is open in $A \backslash \{ 0 \}$.
 \end{proof}
 
\begin{lemma}
\label{lemm: Cas 2}
	Assume that $z_1 \in \partial S_1$ and that $z_1$ is not a corner point but $z_2$ is.
	
	Moreover we suppose that $u_1(z_1) = u_2(z_2) = i(p) = i(q)$ and that $u_1$ has boundary condition along the branch $L_p$ around $z_1$.
	
	Then there is an open neighborhood $\Omega$ of $z_1$ such that $\mathcal{W}(u_1,u_2) \cap \Omega$ is a $\mathcal{C}^1$-embedded graph in $\Omega$.	
\end{lemma}

\begin{proof}
The curve $u_2$ can have two different types of boundary conditions. Accordingly, we will consider two different cases.

\paragraph{\bf{First Case:}} By the Propositions \ref{prop:secondlocalchart}, \ref{prop:localformaroundacornerpoint}, and \ref{prop:local form along a branch} we can assume that 
\begin{enumerate}
	\item the maps $u_1$ and $u_2$ have values in $\C^n$,
	\item the branch $L_p$ is given by $\R^n$ and $L_q$ is given by $e^{i \alpha_1} \cdot \R \times \ldots \times e^{i\alpha_n} \cdot \R$,
	\item there are local $\mathcal{C}^1$ diffeomorphisms $\psi_1$ and $\psi_2$ around $z_1$ and $z_2$ respectively with images $\Omega_1$ and $\Omega _2$ such that
	\[ u_1 \left (\psi_1(z) \right )  =  \lambda_1 A_1 \left (z^p, U_1(z) \right )  \]	
	with $U_1(z) = o(z^p)$,  
	\[ \pi_{\alpha_k} \left (u_2 \circ \psi_2 (z) \right ) = \lambda_2 A_2 \left ( z^{\frac{\alpha_k}{\pi}+m}, \tilde{U}_2(z) \right ) \]
	with $\tilde{U}_2(z) = o \left ( z^{\frac{\alpha_k}{\pi}+m + \delta} \right)$	 and 
	\[ \sum_{\alpha \in \{\alpha_1, \ldots , \alpha_n\} \backslash \{ \alpha_k\}} \pi_\alpha \left (u_2 \circ \psi_2 (z) \right ) = o \left ( z^{\frac{\alpha_k}{\pi}+m + \delta} \right) . \]
\end{enumerate}
Moreover, we can assume that $\mathcal{C}\left (u_1,u_2 \right ) \cap \Omega_1 \subset \{ 0 \}$.

We claim that $A_1(1,0) \in V_{\alpha_k}$ and that there is a $\mu \neq 0$ such that $A_1(1,0) = \mu A_2(1,0)$.

Since $z_1 \mathcal{R}_{u_1}^{u_2} z_2$, there are two sequences $z_{1,\nu} \neq z_1$ and $z_{2,\nu} \neq z_2$ such that $z_{1,\nu} \to z_1$, $z_{2,\nu} \to z_2$ and $u_1(z_{1,\nu}) = u_2(z_{2,\nu})$. Moreover,	 we let $\mu \in \R$ and $v$ be a real vector such that $\pi_{\alpha_k} (A_1(1,0)) = \mu A_2(1,0) + A_2(0,v) $.
If, by contradiction, $\mu = 0$, from the equality
	\[ \lambda_1 z_{1,\nu}^p \pi_{\alpha_k} A_1(1,0) + \lambda_1 \pi_{\alpha_k} A_1(0,U_1(z_{1,\nu})) = \lambda_2 A_2 (z_{2,\nu}^{\frac{\alpha_k}{\pi}+m}, \tilde{U}_2(z_{2,\nu})), \]
we get
	\[ \lambda_1 z_{1,\nu}^p A_2(0,v) + \lambda_1 \pi_{\alpha_k} A_1(0,U_1(z_{1,\nu})) = \lambda_2 A_2 (z_{2,\nu}^{\frac{\alpha_k}{\pi}+m}, \tilde{U}_2(z_{2,\nu})), 	  \]
so $z_{1,\nu}^p = o \left (z_{2,\nu}^{\frac{\alpha_k}{\pi}+m} \right)$. Hence, $u_1(z_{1,\nu}) = o \left (z_{2,\nu}^{\frac{\alpha_k}{\pi}+m} \right ) = o(u_2(z_{2,\nu}))$, which is a contradiction.

In particular, we can deduce that $\mu \lambda_1 z_{1,\nu}^p \sim \lambda_2 z_{2,\nu}^{\frac{\alpha_k}{\pi} + m}$. Denote by $\pi : V_{\alpha_k} \to V_{\alpha_k}$ the (real) orthogonal projection onto $A_2(1,0)^\perp$. Since $A_2$ is orthogonal, we get
\[ \lambda_1 z_{1,\nu}^p A_2(0,v) + \lambda_1 \pi \circ \pi_{\alpha_k} A_1 \left (0,U_1 \left (z_{1,\nu} \right ) \right ) = \lambda_2 \left(0,\tilde{U}_2 \left (z_{2,\nu} \right ) \right).\]
Hence, if $v \neq 0$, we have $z_{1,\nu}^p = o \left (z_{2,\nu} ^{\frac{\alpha_k}{\pi}+m+\delta} \right)$, which is a contradiction.

Assume that $z \in \Omega_1 \cap \mathcal{W}(u_1,u_2)$, then we have $u_1(z) \in \R^n \cup e^{i\alpha_1} \cdot \R \times \ldots \times e^{i\alpha_n} \cdot \R$. Hence, from 
\[ \langle A_1(1,0) , u_1(z) \rangle = z^p , \]
and the fact that $A_1(1,0) \in V_{\alpha_k}$, we deduce that $z^p \in \R \cup e^{i\alpha_k} \cdot \R$.

So $z \in A$ where $A$ is the union of arcs given by
\[ A = \bigcup_{q = 0}^p e^{i \frac{q \pi}{p}} \cdot \R_+ \cup \bigcup_{q = 0}^{E \left (p- \frac{\alpha_k}{\pi} \right )} e^{i \frac{\alpha_k + q \pi}{p}} \cdot \R_+ .  \]

Now we show that the frame $\left ( \mathcal{W}(u_1,u_2) \cap \Omega_1 \right ) \backslash \{0 \}$ is a (possibly empty) union of connected components of $A \backslash \{0 \}$.

Indeed it is closed in $A \backslash \{ 0 \}$ and $\mathcal{W}(u_1,u_2) = \mathcal{R} \left ( \partial S_2 \right )$ is closed.

 Since $\Omega_2 \cap \mathcal{C}(u_1,u_2) \subset \{0\}$ any point of $\mathcal{W}(u_1,u_2) \backslash \{ 0 \}$ is not in $\mathcal{C}(u_1,u_2)$. Therefore, we can apply the proof of \cite[Theorem 3.18]{La11} to conclude that $\mathcal{W}(u_1,u_2)\backslash \{ 0 \}$ is open in $A \backslash \{ 0 \}$.
 
\paragraph{\bf{Second Case:}} This is practically the same as the first case. We explain the differences. This time Proposition \ref{prop: local form of a corner point along the other branch} implies that we can assume 
\[ \pi_{\alpha_k} (u_2 \circ \psi_2(z)) = \lambda_2 e^{i \alpha_k} A_2 \left (z^{\frac{\pi - \alpha_k}{\pi} + m}, U_2(z) \right ) \]
with $U_2(z) = o \left (z^{\frac{\pi - \alpha_k}{\pi} + m \delta} \right )$ and 
\[ \sum_{\alpha \in \{\alpha_1, \ldots, \alpha_n\} \backslash \{\alpha_k \} } \pi_\alpha (u_2 \circ \psi_2(z)) = o \left ( z^{\frac{\pi - \alpha_k}{\pi} + m + \delta}\right). \]

The same argument as in case $1$ shows that $A_1(1,0) \in V_{\alpha_k}$ and that there is a real $\mu \neq 0$ such that $A_1(1,0) = \mu A_2(1,0)$.

Assume that $z \in \Omega_1 \cap \mathcal{W}(u_1,u_2)$. Then from $u_1(z) \in \R^n \cup e^{i \alpha_1} \cdot \R \times \ldots \times e^{i \alpha_n} \cdot \R $ and $\langle A_1(1,0), u_1(z) \rangle = z^p$ we deduce $z^p \in \R \cup e^{i \alpha_k} \R$. Hence $z \in A$.

Now the proof is the same as in the first case.

\end{proof}

\begin{lemma}
\label{lemm: Cas 3}
	Assume that $z_1 \in \partial S_1$ and that $z_1$ is a corner point but $z_2$ is not.
	
	Moreover, we suppose that $u_2$ has boundary condition along the branch $L_p$ around $z_2$.
	
	Then there is an open neighborhood $\Omega$ of $z_1$ such that $\mathcal{W}(u_1,u_2) \cap \Omega$ is a $\mathcal{C}^1$-embedded graph in $\Omega$.	
\end{lemma}

\begin{proof}
Exchanging the roles of $u_1$ and $u_2$, the proof is the same as in Lemma \ref{lemm: Cas 2}.
\end{proof}

\begin{lemma}
\label{lemm: Cas 4}
Assume that $z_1 \in \partial S_1$ and that both $z_1$ and $z_2$ are corner points.
	
Moreover, we suppose that $u_1$ and $u_2$ have boundary condition along $L_p$ followed by $L_q$ around $z_1$ and $z_2$ respectively.
	
Then there is an open neighborhood $\Omega$ of $z_1$ such that $\mathcal{W}(u_1,u_2) \cap \Omega$ is a $\mathcal{C}^1$-embedded graph in $\Omega$.		
\end{lemma}

\begin{proof}
We can assume by Propositions \ref{prop:secondlocalchart} and \ref{prop:localformaroundacornerpoint}, that
\begin{enumerate}
	\item the maps $u_1$ and $u_2$ have values in $\C^n$,
	\item the branch $L_p$ is given by $\R^n$ and $L_q$ is given by $e^{i \alpha_1} \cdot \R \times \ldots \times e^{i \alpha_n} \cdot \R$,
	\item there are local $\mathcal{C}^1$ diffeomorphisms $\psi_1$ and $\psi_2$ around $z_1$ and $z_2$ respectively with images $\Omega_1$ and $\Omega _2$ such that
	\[ \pi_{\alpha_k}(u_1(z)) = \lambda_1 A_1 \left (z^{\frac{\alpha_k}{\pi}+m}, U_1(z) \right ) \]
	with $U_1(z) = o \left (z^{\frac{\alpha_k}{\pi} + m + \delta } \right )$ and
	\[ \sum_{ \alpha \in \{ \alpha_1, \ldots, \alpha_n \} \backslash \{ \alpha_k \} } \pi_\alpha(u_1(z)) = o \left (z^{\frac{\alpha_k}{\pi} + m + \delta } \right ) . \]
	
	Moreover, 
	\[ \pi_{\alpha_l}(u_2(z)) = \lambda_2 A_2 \left (z^{\frac{\alpha_l}{\pi}+p}, U_2(z) \right )  \]
	with $U_2(z) = o \left (z^{\frac{\alpha_l}{\pi} + p + \delta } \right)$ and
	\[ \sum_{ \alpha \in \{ \alpha_1, \ldots, \alpha_n \} \backslash \{ \alpha_l \} } \pi_\alpha(u_2(z)) = o \left (z^{\frac{\alpha_l}{\pi} + p + \delta } \right ) . \]
	
\end{enumerate}

Furthermore, there are two non-zero sequences $(z_{1,\nu})$ and $(z_{2,\nu})$ which converge to $0$ such that $u_1(z_{1,\nu}) = u_2(z_{2,\nu})$. 

First we can easily see that $\alpha_k = \alpha_p$. Assume the opposite. Then
\[  o \left (z_{2,\nu}^{\frac{\alpha_l}{\pi}+p +\delta} \right ) = \pi_{\alpha_k} (u_2(z_{2,\nu})) = \pi_{\alpha_k} (u_1(z_{1,\nu})) = \lambda_1 A_1 \left (z_{1,\nu}^{\frac{\alpha_k}{\pi}+m }, U_1(z_{1,\nu}) \right) . \]
So $z_{1,\nu}^{\frac{\alpha_k}{\pi}+m} = o \left (z_{2,\nu}^{\frac{\alpha_l}{\pi}+p +\delta} \right )$. Exchanging the roles of $\alpha_k$ and $\alpha_l$, we get that $z_{2,\nu}^{\frac{\alpha_l}{\pi}+p} = o \left (z_{1,\nu}^{\frac{\alpha_k}{\pi}+m+\delta} \right )$, a contradiction.

As usual, we claim that there is $\mu \in \R \backslash \{ 0 \}$ such that $A_1(1,0) = \mu A_2(1,0)$.

Let $\mu \in \R$ and $v$ be a vector such that $A_1(1,0) = \mu A_2(1,0) + A_2(0,v) $. Assume by contradiction that $\mu = 0$. Then since
\[ \lambda_1 z_{1,\nu}^{\frac{\alpha_k}{\pi}+m} A_2(0,v) + \lambda_1 A_1(0,U_1(z_{1,\nu})) = \lambda_2 A_2 \left (z_{2,\nu}^{\frac{\alpha_l}{\pi}+p}, U_2(z_{2,\nu}) \right ), \]
and since $A_2$ is an isometry, we get $z_{1,\nu}^{\frac{\alpha_k}{\pi}+m } = o \left (z_{2,\nu}^{\frac{\alpha_l}{\pi}+p} \right)$. Hence, $u_1(z_{1,\nu}) = o(u_2(z_{2,\nu}))$, a contradiction.

In particular $\mu \lambda_1 z_{1,\nu}^{\frac{\alpha_k}{\pi}+m} \sim \lambda_2 z_{2,\nu}^{\frac{\alpha_l}{\pi}+p}$. Moreover, applying the (complexified) orthogonal projection onto $A_2(1,0)^\perp$, we get
\[  \lambda_1 z_{1,\nu}^{\frac{\alpha_k}{\pi}+m} A_2(0,v) + \lambda_1 A_1(0,U_1(z_{1,\nu})) = o \left (z_{2,\nu}^{\frac{\alpha_l}{\pi}+p + \delta} \right ) . \]
This implies $A_2(0,v) = 0$.

Now assume that $z \in \mathcal{W}(u_1,u_2)$. Then $u_1(z) \in \R^n \cup e^{i\alpha_1} \cdot \R \times \ldots \times e^{i\alpha_n} \cdot \R$. As usual, we take the scalar product of $\pi_{\alpha_k} u_1(z)$ with $A_1 (1,0)$ to obtain that $z^{\frac{\alpha_k}{\pi}+m} \in \R \cup e^{i \alpha_k} \cdot \R$. Therefore $z \in A$, where $A$ is the set
\[ A := \left ( \bigcup_{q = 0 }^m e^{i \frac{\alpha_k + q \pi}{\frac{\alpha_k}{\pi} + m }} \cdot \R_+ \right ) \cup \left ( \bigcup_{q=0}^{E( \frac{\alpha_k}{\pi} + m )}  e^{i \frac{q \pi}{\frac{\alpha_k}{\pi} + m}} \cdot \R_+ \right ) . \]

We now show that the frame $ ( \mathcal{W}(u_1,u_2) \cap \Omega_1 ) \backslash \{0\}$ is a (possibly empty) union of connected components of $A \backslash \{ 0 \}$.

It is closed in $A \backslash \{ 0 \}$ since $\mathcal{W}(u_1,u_2)$ is closed.

 Since $\Omega_2 \cap \mathcal{C}(u_1,u_2) \subset \{0\}$ any point of $\mathcal{W}(u_1,u_2) \backslash \{ 0 \}$ is not in $\mathcal{C}(u_1,u_2)$. Therefore, we can apply the proof of \cite[Theorem 3.18]{La11} to conclude that $\mathcal{W}(u_1,u_2)\backslash \{ 0 \}$ is open in $A \backslash \{ 0 \}$.
\end{proof}

\begin{lemma}
\label{lemm: Cas 5}
Assume that $z_1 \in \Int(S_1)$ and that $z_2$ is a corner point.

Then there is an open neighborhood $\Omega$ of $z_1$ such that $\mathcal{W}(u_1,u_2) \cap \Omega$ is a $\mathcal{C}^1$-embedded graph in $\Omega$.	
\end{lemma}

\begin{proof}
Using Propositions \ref{prop:secondlocalchart}, \ref{prop:localformaroundacornerpoint} and \ref{prop: local behavior of a curve around an interior point}, we can assume that
\begin{enumerate}
	\item the maps $u_1$ and $u_2$ have values in $\C^n$,
	\item the branch $L_p$ is given by $\R^n$ and the branch $L_q$ is given by $e^{i\alpha_1} \cdot \R \times \ldots \times e^{i\alpha_n} \cdot \R$,
	\item there are local $\mathcal{C}^1$-diffeomorphisms $\psi_1$ and $\psi_2$ around $z_1$ and $z_2$ respectively with images $\Omega_1$ and $\Omega_2$ such that 
	\[ u_1( \psi_1 (z) ) = \lambda_1 A_1 \left (z^k,U_1(z) \right ), \]
	with $U_1(z) = O \left (z^{k+1} \right )$ and 
	\[ \pi_{\alpha_l} \circ u_2( \psi_2(z)) = \lambda_2 A_2 \left (z^{\frac{\alpha_l}{\pi}+m}, U_2(z) \right ), \]
	with $U_2(z) = o \left ( z^{\frac{\alpha_l}{\pi}+m+\delta} \right )$ and
	\[ \sum_{\alpha \in \{ \alpha_1, \ldots , \alpha_n \} \backslash \{ \alpha_l \}} \pi_\alpha \circ u_2( \psi_2(z)) =  	o \left ( z^{\frac{\alpha_l}{\pi}+m+\delta} \right ) . \]
\end{enumerate}
Replacing $\Omega_1$ and $\Omega_2$ by smaller neighborhoods if necessary, we can assume that $\mathcal{C}(u_1,u_2) \cap \Omega_1 \subset \{ 0 \}$.

Moreover, since $z_1 \mathcal{R}_{u_1}^{u_2} z_2$, there are non-zero sequences $(z_{1,\nu})$ and $(z_{2,\nu})$ converging to $0$ such that $u_1(\psi_1(z_{1,\nu})) = u_2(\psi_2(z_{2,\nu}))$.

First, we easily see that for $p \neq l$, $\pi_{\alpha_p} A_1(1,0) = 0$. Indeed, assume by contradiction that there is $p$ such that $\pi_{\alpha_p} A_1(1,0) \neq 0$. Apply the projection $\pi_{\alpha_p}$ to the equality above to get
\[ o \left (z_{2,\nu}^{\frac{\alpha_l}{\pi}+m} \right ) = \lambda_1 z_{1,\nu}^k \pi_{\alpha_p} A_1(1,0) +o \left (z_{1,\nu}^k \right) . \]
So necessarily $z_{1,\nu}^k = o \left (z_{2,\nu}^{\frac{\alpha_l}{\pi}+m} \right ) $ and $u_2(z_{2,\nu}) = u_1(z_{1,\nu}) = o \left (z_{2,\nu}^{\frac{\alpha_l}{\pi}+m} \right )$, a contradiction.

Now, we claim that there is a $\lambda \in \C \backslash \{0 \}$ such that $A_1(1,0) = \lambda A_2(1,0)$. To see this, let $\lambda \in \C$ and $v$ be a complex vector such that
\[ \pi_{\alpha_l} (A_1(1,0)) = \lambda A_2(1,0) + A_2(0,v) . \]
If $\lambda = 0$, we apply the complexification of the real scalar product with $A_2(1,0)$ to get
\[ o \left ( z_{1,\nu}^k \right ) = \lambda_2 z_{2,\nu}^{\frac{\alpha_l}{\pi}+m} + o \left (z_{2,\nu}^{\frac{\alpha_l}{\pi}+m} \right ). \]
so $z_{2,\nu}^{\frac{\alpha_l}{\pi}+m} = o \left (z_{1,\nu}^k \right )$ and $u_2(z_{2,\nu}) = o(u_1(z_{1,\nu}))$, a contradiction.

In particular, we have $\lambda \lambda_1 z_{1,\nu}^k \sim \lambda_2 z_{2,\nu}^{\frac{\alpha_l}{\pi}+m}$. 

Assume that $\pi$ is the real orthogonal projection onto $A_2(1,0)^\perp$. We apply its complexification to get
\[ z_1^k \lambda \lambda_1 A_2(0,v) = o\left ( z_{2,\nu}^{\frac{\alpha_k}{\pi}+m} \right ). \]
This shows that $A_2(0,v) = 0 $. 

Now assume that $z \in \mathcal{W}(u_1,u_2) \cap \Omega_1$, so $u_1(z) \in \left (e^{i\alpha_1} \cdot \R \times \ldots e^{i\alpha_n} \cdot \R \right ) \cup \R^n $. Denote by $h_{\text{std}}$ the standard hermitian scalar product (which is complex linear in the first variable). Then, since $A_1 \in U(n)$,
\begin{eqnarray*}
	 h_{\text{std}} \left (u_1(z), A_1(1,0) \right ) & = &  h_{\text{std}} \left (\lambda_1 z^k A_1(1,0) + A_1 \left (0,U_1(z) \right ) , A_1(1,0) \right ) \\
	 & = & \lambda_1 z^k .         
\end{eqnarray*}
Recall that $\lambda^{-1} A_1(1,0) \in V_{\alpha_l}$. For $v \in e^{i\alpha_1} \cdot \R \times \ldots e^{i \alpha_n} \cdot \R$, we have $h_{\text{std}} \left (v,\frac{A_1(1,0)}{\lambda} \right ) \in e^{i\alpha_l} \cdot \R$. For $v \in \R^n$, $h_{\text{std}} \left (v,\frac{A_1(1,0)}{\lambda} \right ) \in \R$.

In the end, we conclude that $z^k \in \overline{\lambda} \cdot\R  \cup \overline{\lambda} e^{i\alpha_l} \cdot \R$. Call $\theta_\lambda \in [0,\pi] $ an argument of $\overline{\lambda}$ (resp. $- \overline{\lambda}$) if $Im(\overline{\lambda}) > 0$ (resp. $Im(\overline{\lambda}) < 0$). Then $z \in A$ where $A$ is a union of half-rays with extremities at $0$,
\[ A := \bigcup_{p=0}^{2k} e^{i \frac{p\pi}{k}} \cdot \R_+ \cup \bigcup_{p = E \left (-\frac{\theta_\lambda + \alpha_l}{\pi} \right )}^{E \left (2k - \frac{\theta_\lambda + \alpha_l}{\pi} \right )}  e^{ i  \frac{p\pi + \theta_\lambda + \alpha_l}{k}} \cdot \R_+ . \]
Now we show that the frame $\left ( \mathcal{W}(u_1,u_2) \cap \Omega_1 \right ) \backslash \{0 \}$ is a (possibly empty) union of connected components of $A \backslash \{0 \}$.

Indeed it is closed in $A \backslash \{ 0 \}$ since $\mathcal{W}(u_1,u_2) = \mathcal{R}_{u_1}^{u_2} \left ( \partial S_2 \right )$ is closed.

 Since $\Omega_2 \cap \mathcal{C}(u_1,u_2) \subset \{0\}$ any point of $\mathcal{W}(u_1,u_2) \backslash \{ 0 \}$ is not in $\mathcal{C}(u_1,u_2)$. Therefore, we can apply the proof of \cite[Theorem 3.18]{La11} to conclude that $\mathcal{W}(u_1,u_2)\backslash \{ 0 \}$ is open in $A \backslash \{ 0 \}$.
\end{proof}

With the cases already proved in \cite{La11}, we readily conclude that the following proposition holds.

\begin{prop}
\label{prop:theframeisagraph}
	Let $\mathcal{D}(u_1, u_2)$ be the set of isolated points of $\mathcal{W}(u_1,u_2)$. The following is true
	\begin{enumerate}[label=(\roman*)]
		\item $\mathcal{D}(u_1,u_2) \subset \mathcal{C}(u_1,u_2) \cap \partial S_1$,
		\item $\mathcal{W}(u_1,u_2) \backslash \mathcal{D}(u_1,u_2)$ is a $\mathcal{C}^1$-embedded graph in $S_1$, its vertices are in $\mathcal{C}(u_1,u_2)$.
		\item $\overline{\mathcal{W}} (u_1,u_2)$ is a $\mathcal{C}^1$-embedded graph.	
	\end{enumerate}

\end{prop}

\begin{proof}
First, note that \textit{(iii)} follows immediately from \textit{(i)} and \textit{(ii)}. We shall see that \text{(ii)} follows quite easily form the lemmas proved above.

Let $z \in \mathcal{W}(u_1,u_2) \cap \Int(S_1)$ so there is $z_2 \in \partial S_2$ such that $z \mathcal{R}_{u_1}^{u_2}z_2$. Assume first that $z_2$ is not a corner point. It follows from the proof of \cite[Theorem 3.18]{La11} that $\mathcal{W}(u_1,u_2)$ is a $\mathcal{C}^1$-graph around $z$. The same results holds if $z_2$ is a corner point, this is the content of Lemma \ref{lemm: Cas 5}.

We prove, using that the frame relation is open in some cases (Proposition \ref{prop:Theframerelationisopen}), that $z$ is not isolated in $\mathcal{W}(u_1,u_2)$. To see this, pick an open neighborhood $V_2$ of $z_2$ such that $V_2 \cap u_2^{-1}(z_2) = \{ z_2 \}$ and a decreasing sequence of neighborhoods $V_{1,\nu} \subset \Int(S_1)$ such that $ \{ z \} = \cap_\nu V_{1,\nu}$. For $\nu \in \N$, the projection $ \left ( V_{1,\nu} \times V_2 \right ) \cap \mathcal{R}_{u_1}^{u_2} \to V_2$ is open (since $V_{1,\nu} \subset \Int(S_1) $). Hence, there are $z_{2,\nu} \neq z_2 \in \partial S_2$ and $z_{1,\nu} \in V_{1,\nu}$ with $z_{1,\nu} \mathcal{R}_{u_1}{u_2} z_{2,\nu}$. Necessarily $z_{1,\nu} \neq z_1$ (otherwise $u_1(z_1) = u_2(z_{2,\nu})$ which would yield $z_2 \in u_2^{-1}(u_2(z_2))$). We conclude that $z$ is an accumulation point of $\mathcal{W}(u_1,u_2)$.

Now assume that $z \in \mathcal{W}(u_1,u_2) \cap \partial S_1$ and and that $z \notin \mathcal{C}(u_1,u_2)$. Pick $z_2$ such that $z \mathcal{R}_{u_1}^{u_2} z_2$. In particular $z_2$ is not a corner point, $du_1(z) \neq 0$ and $du_2(z_2) \neq 0$. We can apply \cite[Proposition 3.13]{La11} : there are open neighborhoods $\omega_1$ and $\omega_2$ of $z$ and $z_2$ respectively such that $\phi(\omega_2 \cap \partial S_2) = \omega_1 \cap \partial S_1$ and $z \mathcal{R}_{u_1}^{u_2} z' $ if and only if $z = \phi(z')$. Therefore $\omega_1 \cap \mathcal{W}(u_1,u_2) = \partial S_1 \cap \omega_1$ : the frame is a local $\mathcal{C}^1$-graph around $z$ and $z$ is not isolated in $\mathcal{W}(u_1,u_2)$.

Assume that $z$ is not a corner point, then Lemma \ref{lemm: Cas 2} if $z_2$ is a corner point, and Lemma \ref{lemm: Cas 1} if $z_2$ is not, show that $\mathcal{W}(u_1,u_2)$ is a graph around $z$.

If $z$ is a corner point, Lemma \ref{lemm: Cas 3} if $z_2$ is not a corner point, and Lemma \ref{lemm: Cas 4} if $z_2$ is, show again that $\mathcal{W}(u_1,u_2)$ is a graph around $z$.
\end{proof}

\subsection{Frame and simple curves}
\subsubsection{Lifts of curves}
Two simple curves which have the same images are reparameterizations of each other through a biholomorphism (\cite[Section 4]{La11}). In this section, we shall recall the statement of these results as well as provide proofs when needed.

The relation that we just defined is not quite transitive. However, there still are a few cases where transitivity holds. Let us consider three $J$-holomorphic curves with boundary in $L$, $u_i : (S_i, \partial S_i) \to (M, i(L) )$ with $i = 1 \ldots 3$.

\begin{prop}[Proposition 4.1, \cite{La11}]
\label{prop:transitivityoftherelation}
If $(z_1, z_2, z_3) \in S_1 \times S_2 \times S_3 $ satisfy $z_1 \mathcal{R}_{u_1}^{u_2} z_2$ and $z_2 \mathcal{R}_{u_2}^{u_3} z_3$, and one of the following holds
	\begin{enumerate}
		\item $z_1 \in \Int(S_1)$ or $z_3 \in \Int(S_3)$,
		\item $z_2 \in \partial S_2$ and there is a neighborhood $\omega_2 \subset S_2$ of $z_2$ such that $\mathcal{W}(u_1, u_2) \cap \omega_2 \subset \partial S_2$ or $\mathcal{W}(u_2, u_3) \cap \omega_2 \subset \partial S_2$,
	\end{enumerate}

then $z_1 \mathcal{R}_{u_1}^{u_3} z_3$. 	
\end{prop}

Moreover, it turns out that the relation $\mathcal{R}$ has the lifting property with respect to the projection on the second factor.

\begin{prop}[Lemma 4.3, \cite{La11}]
\label{prop:lift1}
	Let $z_1 \in S_1 \backslash \left ( \mathcal{W}(u_1, u_2) \cup \mathcal{C}(u_1,u_2) \right )$ and $z_2 \in S_2$ such that $z_1 \mathcal{R}_{u_1}^{u_2} z_2$.
	
	Assume that $\gamma_1 : [0,1] \to S_1$ is a continuous map such that $\gamma_1(0) = z_1$ and for $t \in [0,1[$, $\gamma_1(t) \notin \mathcal{W}(u_1,u_2) \cup \mathcal{C}(u_1,u_2)$.
	
	There exists a unique continuous map $\gamma_2 : [0,1] \to S_2$ such that $\gamma_2(0) = z_2$ and for $t \in [0,1]$, $\gamma_1(t) \mathcal{R}_{u_1}^{u_2} \gamma_2(t)$. 
\end{prop}

\begin{prop}[Lemma 4.4, \cite{La11}]
\label{prop:lift2}
	Let $\gamma_1 : [0.1] \to \overline{\mathcal{W}} (u_1,u_2)$ be a continuous path such that $\gamma_1(t) \notin \mathcal{C}(u_1,u_2)$ for $0 < t < 1 $. If $z_1 = \gamma_1(0)$ and $z_1 \mathcal{R}_{u_1}^{u_2} z_2$ with $z_2 \in \Int(S_2)$. 
	
	There is $\gamma_2 : [0,1] \to \mathcal{W}(u_2 ; u_1, u_2)$ such that $\gamma_2(0) = z_2$ and $\gamma_1(t) \mathcal{R}_{u_1}^{u_2} \gamma_2(t)$ for $0 \leqslant t \leqslant 1$.
	
	Moreover, if $\gamma_1([0,1]) \subset \mathcal{W}(u_1,u_2)$, then $\gamma_2([0,1]) \subset \mathcal{W}(u_2,u_2)$. Otherwise $\gamma_2([0,1]) \subset \mathcal{W}(u_2,u_1)$.
\end{prop} 

\begin{prop}[Lemma 4.5, \cite{La11}]
	Assume $\partial S_2$ is connected. If $C$ is a connected component of $\mathcal{W}(u_1,u_2)$ with $C \subset \Int(S_1)$, then for $z \in \partial S_2$ there is $w \in C$ such that $w \mathcal{R}_{u_1}^{u_2} z$.
\end{prop}

We also introduce the following notion.

\begin{defi}
	The $J$-holomorphic curve $u : (S, \partial S) \to (M, i(L))$ is \emph{properly bordered} if $\partial S$ is open in $\mathcal{W}(u,u)$.
\end{defi}

\subsubsection{Simple curves are determined by their frames}
	\begin{defi}
		Two $J$-holomorphic curves with boundaries on $L$ defined on connected surfaces $u_1$ and $u_2$ have relatively simple frames if $\mathcal{R}_{u_1}^{u_2} \neq \emptyset$, $\mathcal{W}(u_1,u_2) \subset \partial S_1$ and $\mathcal{W}(u_2,u_1) \subset S_2$. 
	\end{defi}
	
The argument of \cite[section 4]{La11} holds without modifications to yield

\begin{Theo}[Theorem 4.13, \cite{La11}]
\label{Theo:Relatively simple frames implies simple curves are equal}
	If $u_1$ and $u_2$ are two simple curves with boundary in $L$, the following assertions are equivalent
	\begin{enumerate}
		\item $u_1(S_1) = u_2(S_2)$ and $u_1(\partial S_1) = u_2(\partial S_2)$,
		\item $\mathcal{W}(u_1, u_2) = \partial S_1$ and $\mathcal{W}(u_2,u_1) = \partial S_2$,
		\item $u_1$ and $u_2$ have relatively simple frames,
		\item $\mathcal{R}_{u_1}^{u_2} \neq \emptyset$ and $u_1(\partial S_1) = u_2(\partial S_2)$,
		\item There exist biholomorphism $\phi_{12} : (S_2, \partial S_2) \to (S_1, \partial S_1)$ such that 
		\[ u_2 = u_1 \circ \phi_{12}. \] 	
	\end{enumerate}
\end{Theo}

\subsection{Factorizations of $J$-holomorphic disks}
	\subsubsection{Factorization of curves}

\begin{prop}
\label{prop : biholomorphisms at the boundary }
	Let $u :(S, \partial S) \to (M, i(L))$ be a non-constant finite-energy $J$-holomorphic curve with corners and boundary on $L$. We suppose that $\mathcal{W}(u)$.
	
	Let $\{ z_1, \ldots , z_N \} \subset \partial S$ such that $\mathcal{R}_u^u \{z_1\} = \{ z_1, \ldots , z_N\} $. 
	
	There are simply connected open sets $\Omega_i \ni z_i$ for $i = 1, \ldots , N$ such that $\Omega_i \cap \mathcal{C}(u,u) \subset \{ z_i \}$ and applications $\psi_{ij} :  \overline{\Omega_i} \to \overline{\Omega_j} $ such that
	\begin{enumerate}
		\item $\psi_{ij}$ is the unique biholomorphism such that $u \circ \psi_{ij} = u$,
		\item If $(z,w) \in \overline{\Omega_i} \times \overline{\Omega_j}$, we have $z \mathcal{R}_u^u w$ if and only if $w = \psi_{ij}(z)$.	
	\end{enumerate}	
\end{prop}

\begin{proof}
For $i= 1 \ldots N $, choose $V_i$ a neighborhood of $z_i$ and $\mathcal{C}^1$ charts such that, in these charts,
	\[ u(z) = a_i (z-z_i)^{k_i} + o(\norm{z-z_i}^{k_i})\]
	if $z_i$ is not a corner point, or
	\[ u(z) = a_i(z-z_i)^{\frac{\alpha_i}{\pi} + m_i} + o(\norm{z-z_i}^{\frac{\alpha_i}{\pi} + m_i}) \]
if $z_i$ is.

Moreover, we can assume that $d \norm{u}_z$ is non-zero on $V_i$. We choose an $\alpha \in$ such that $0 < \alpha < \inf_{\partial(\cup V_i) \backslash \partial \D}$. Then $\alpha$ is a regular value of $\norm{u}$. Denote by $\Omega_i$ the connected component of $\D^+ \backslash \norm{u}^{-1} \{ \alpha \}$ such that $\Omega_i \ni  z_i$.
Since
\[ d \norm{u} (z - z_i) = \norm{a_i}^2 \left (\frac{\alpha_i}{\pi} +m_i \right ) + o(1) \]
is positive for $z$ close enough to $z_i$, we can assume (choosing $\alpha$ smaller if necessary) that $\Omega_i$ is simply connected. It implies that it is biholomorphic to a disk. Furthermore, its boundary is the union of an embedded arc in $\partial S$ and an embedded arc in the interior.

Let us show that the $\psi_{ij}$ exist. Choose $\tilde{z_1} \in \Omega_1 \cap \Int(S)$ and $\tilde{z_2} \in \Omega_2 \cap \Int(S)$. We build $\psi_{12} : \Omega_1 \backslash \partial \D \to \Omega_2 \backslash \partial \D $.

For this, choose $z \in \Omega_1 \backslash \partial \D $. There is a continuous path $\gamma : [0,1] \to \Omega_1 \backslash \partial \D$ from $\tilde{z}_1$ to $z$. This path lifts to a unique continuous $\gamma_2 : [0,1] \to S$ such that $\gamma_2(0) = \tilde{z}_2$ and $\gamma(t) \mathcal{R}_{u_1}^{u_2} \gamma_2(t)$. Notice that since $\mathcal{W}(u) =  \partial S$, we have $\gamma_2(t) \notin \partial S$. Moreover, since $\norm{u(\gamma(t))} < \alpha$ for all $t \in [0,1]$ and $u(\gamma_2(t))= u(\gamma(t)) $, we get $\gamma_2(t) \in \Omega_2 \backslash S$. We put $\psi_{12}(z) = \gamma_2(1) $.

It remains to see that $\psi_{12}(z)$ does not depend on the choice of $\gamma$. For this, suppose that there is a homotopy $H : [0,1] \times [0,1] \to \Omega_1 \backslash \partial \D $ such that $H(0, \cdot) = \gamma$. By the same argument as before, there is a unique lift $\tilde{H} : [0,1] \times [0,1] \to \Omega_2 \backslash \partial \D$ such that $H(s,t) \mathcal{R}_u^u \tilde{H}(s,t)$. It is easy to see, using that there is no critical points of $u$ in $\Omega_2 \backslash \partial S$, that $H$ is actually smooth. Thus the existence of $\psi_{12}$ is proved. 

Recall that there are no critical points of $u$ in $\Omega_2 \backslash \{ z_2 \}$. It is easy to see, using the same argument, that $\psi_{12}$ extends to a holomorphic map $\Omega_1 \backslash \{ z_1 \} \to \Omega_2 \backslash \{ z_2 \}$. In a local chart for $S$, this is a bounded holomorphic map from $\D^+_\R \backslash \{ 0 \}$ to $\D^+_\R \backslash \{ 0 \}$ sending the real line to the real line. Hence, it extends to a holomorphic map $\Omega_1 \to \Omega_2 $.

To see that this is a biholomorphism, notice that the same argument allows us to build a holomorphic map $\psi_{21} : \Omega_2 \to \Omega_1$ such that $\psi_{21}(\tilde{z}_2) = \tilde{z}_1$. Now $\Phi = \psi_{12} \circ \pi_{21} : \Omega_1 \to \Omega_1$ satisfies $u \circ \Phi = u$ and $\Phi(\tilde{z}_1) = \tilde{z}_1 $. Hence by the unicity of lifts of paths, it is the identity. Exchanging the order of the composition, we get $\psi_{12} \circ \psi_{21} = \Id$.

The unicity follows from a beautiful argument given in \cite[Proposition 5.9]{La00}. Assume that $\psi$ is a biholomorphism $\Omega_1 \to \Omega_1$ such that $u \circ \psi  = u $. Since $\Omega_1$ is simply connected, we can assume that it is actually equal to $\D$ by the Riemann mapping theorem. Then $\psi$ has a fixed point, say $z_0$. If $z_0 \in \partial \D$, then $\psi^n(z) \to z_0$ for all $z$, so $u$ is constant. Hence $z_0 \in \mathring{\D}$ and we can assume $z_0 = 0$ so $\psi(z) = \zeta z$ for some $\zeta \in \partial \D$. Either $\zeta$ has infinite order, in which case $u$ is constant, or it has finite order so $u$ factors through $z \mapsto z^k$. But $u$ is an immersion on $\D$, so $k = 1$. 
\end{proof}

From this, we immediately get the following corollary.

\begin{coro}
\label{coro: charts at the boundary of the quotient}
	If $\mathcal{W}(u)= \partial S$ and $z_1, \ldots, z_N \in \partial S$ are such that $\mathcal{R}_{u}^u \{z_1 \} = \{ z_1, \ldots , z_n \}$. There are holomorphic charts $h_i : (\D^+, \D^+_\R) \to (S, \partial S)$ with $h_i(0) = z_i$ such that $h_i(z) = h_j(z')$ if and only if $z \mathcal{R}_u^u z'$.	
\end{coro}

\begin{coro}
\label{coro: multiply covered disks}
Assume that $\mathcal{W}(u) = \partial S$, then there is a simple, finite-energy $J$-holomorphic curve $v :(S', \partial S') \to (M, i(L) )$ and a finite branched cover $p: (S, \partial S) \to (S', \partial S') $ which restricts to an actual cover $p : \partial S \to \partial S'$ such that
	\[ u = v \circ p . \] 	
\end{coro}

\begin{proof}
The relation $\mathcal{R}_u^u$ is transitive by Proposition \ref{prop:transitivityoftherelation}. Therefore, the quotient $S' = S / \mathcal{R}_u^u$ is well-defined.

It remains to define holomorphic charts on $S'$ such that the quotient map $p : S \to S'$ is holomorphic. 

Let $z' \in S'$, let 
	\[ p^{-1}(z) = \{ z_1, \ldots , z_N \} \subset \partial S \]
be the preimages of $z'$. We consider the biholomorphisms $h_1, \ldots, h_N$ given by Corollary \ref{coro: charts at the boundary of the quotient}. Then the restriction of $p$ to each $\Omega_i$ is a bijection. Remark that, up to taking smaller neighborhoods, the $\Omega_i$ can be assumed relatively compact. Therefore, we can assume that $p$ restricted to each $\Omega_i$ is a homeomorphism onto its image. The chart around $z'$ is given by $h_i \circ p_{\vert \Omega_i}^{-1}$.

If the preimages of $z'$ are contained in $\Int(S)$, the charts are constructed in \cite[Proposition 2.5.1]{McDS12}.

Notice that this immediately implies that $p_{\vert \partial S} : \partial S \to \partial S'$ is a finite cover.

The map $u$ goes through the quotient to induce a holomorphic map $v : (S', \partial S') \to (M, i(L))$ which is simple. Call $E = \{ y_1, \ldots , y_m \}$ the corner points of $u$ and let $\gamma : \partial S \backslash E \to L $ be the boundary condition of $u$. Corollary \ref{coro: charts at the boundary of the quotient} immediately implies that if $z \in \partial S$ is not a corner point and $z'$ is such that $z \mathcal{R}_u^u z'$, then $z'$ is not a corner point. We deduce that in a neighborhood of $z$, we have $\gamma \circ h_{z} = \gamma \circ h_{z'}$. Hence, $\gamma$ can be quotiented out to give a continuous map $\gamma' : \partial S' \backslash p(E)$ which satisfies $i \circ \gamma' = v_{\vert S' \backslash p(E)}$.

It is now immediate from the definition of $\gamma'$ that each point of $p(E)$ is a corner point.

Since $p$ is a branched cover of finite degree, say $d \geqslant 1$, we have
	\[ \int u^{*} \omega = d \int v^{*} \omega ,\]
	so $\int v^{*} \omega < + \infty$.
	
The fact that $v$ is simple is an easy consequence of the definition of $S'$. If $z \mathcal{R}_v^v z'$, there are two sequences $z_\nu \to z$ and $z_\nu' \to z' $ such that $z_\nu \notin \mathcal{C}(v)$, $z_\nu' \notin \mathcal{C}(v)$ and $z_\nu \mathcal{R}_v^v z_\nu'$. Now pick two sequences of lifts of these $\tilde{z}_\nu, \tilde{z}_\nu'$ which converge (up to a subsequence) to two points say $\tilde{z}$ and $\tilde{z}'$. Notice that $p$ is a branched cover, hence a local embedding outside the critical points. With this in mind, $z_\nu \mathcal{R}^v_v z_\nu'$ implies $\tilde{z}_\nu \mathcal{R}_u^u \tilde{z}_\nu'$. This implies in turn $\tilde{z} \mathcal{R}_u^u \tilde{z}'$. So by definition $z = p(\tilde{z}) = p(\tilde{z}')$. Hence, the relation $\mathcal{R}_v^v$ is trivial. Now apply Proposition \ref{prop:Characterizationofsimplecurves}.
\end{proof}

We can conclude that any curve can be decomposed into simple pieces.

\begin{Theo}
\label{Theo:decomposition générale en courbes simples}
	Let $u : (S, \partial S) \to (M, i(L))$ be a finite-energy $J$-holomorphic curve with boundary in $L$. There are finite-energy simple $J$-holomorphic curves with corners $v_i : (S_i, \partial S_i) \to (M, i(L)) $ for $i = 1 \ldots N$ such that 
	\[ \Im(u) = \bigcup_{i= 1}^ {N} \Im(v_i) \] 
	
	and there are integers $m_1, \ldots , m_N \geqslant 1$ such that
	\[ [u] = \sum_{i=1}^N m_i [v_i] \text{ in } H_2(M,i(L)).\] 
\end{Theo}

\begin{proof}
The proof of the first point proceeds as in Lazzarini's paper. For each connected component $\Omega$ of $S \backslash \mathcal{W}(u)$, choose a complex embedding $h_\Omega : (S_\Omega, \partial S_\Omega) \to (\Omega, \partial \Omega)$ (\cite[Lemma 2.6]{La11}) and consider the map $u \circ h_\Omega$.
	
We have 
	\[ E(u \circ h_\Omega) \leqslant E(u_{\vert \Omega}) \leqslant E(u) < + \infty .\]
The set of the preimages of double points $u^{-1} \left (i(R) \right)$ is finite, hence by \cite[Lemma 2.4]{La11}, the set $\left (u \circ h_\Omega \right )^{-1}(i(R)) $ is also finite. Therefore $u \circ h_\Omega$ has a finite number of corner points.
	
Now we claim that $\mathcal{W} \left ( u \circ h_\Omega \right ) = \partial S_\Omega$. To see this, let $z \in \mathcal{W} \left (u \circ h_\Omega \right)$. There is $z' \in \partial S_\Omega$ such that $z' \mathcal{R}_{u \circ h_\Omega}^{u \circ h_\Omega} z$.	 From this, it follows that $h_\Omega(z) \mathcal{R}_u^u h_\Omega(z')$. If $h_\omega(z') \in \Int(S)$, from $h_\Omega(z') \in \mathcal{W}(u)$ it follows that $h_\Omega(z) \in \mathcal{W}(u)$ by transitivity. Hence, $z \in \partial S_\Omega$. Therefore, $\mathcal{W} \left ( u \circ h_\Omega \right ) \subset \partial S_\Omega$ and there is equality since the other inclusion holds by definition.
	
By Corollary \ref{coro: multiply covered disks}, there is a Riemann surface with boundary $S_\Omega'$, a map $p_\Omega: S_\Omega \to S_\Omega '$ and a simple curve $v_\Omega : S_\Omega' \to (M,i(L))$ such that 
		\[ u \circ h_\Omega = v_\Omega \circ p_\Omega .\]		
Moreover, we see immediately that 
	\[ \Im (u) = \bigcup_{\Omega} \Im(v_\Omega) , \]
where the union is taken over the set of connected components of $\D \backslash \mathcal{W}(u)$. 	
\end{proof}

Now the conclusion of the main Theorem \ref{Theo:Maintheorem} follows immediately from the following proposition.

\begin{prop}
\label{prop:les pièces de la décomposition sont des disques}
Assume that $u : (\D, \partial \D) \to (M, L)$ is a finite-energy $J$-holomorphic disk with corners and boundary on $L$. Keeping the notations of the proof of Theorem \ref{Theo:decomposition générale en courbes simples}, each of the surfaces $S_\Omega'$ is biholomorphic to a disk. 	
\end{prop}

Given what we have already shown, the proof of Proposition \ref{prop:les pièces de la décomposition sont des disques} does not differ much from the proof of the corresponding proposition in \cite[Proposition 5.5]{La11}. For the convenience of the reader, we shall recall the proof in the next subsection.
 
\subsubsection{Connectedness of the frame and holomorphic spheres}

The main result is the following proposition.
\begin{prop}
\label{prop: frame pas connexe implique sphère}
Assume that $u : (\D, \partial \D) \to (M,L)$ is a finite-energy $J$-holomorphic disk with corners with $\mathcal{W}(u)$ not connected. There is a simple $J$-holomorphic sphere $v : \C P^1 \to M $ such that $\Im(u) = \Im(v)$ and $\mathcal{R}^u_v \neq \emptyset$. 	
\end{prop}

\begin{proof}[Proof of Proposition \ref{prop:les pièces de la décomposition sont des disques} assuming \ref{prop: frame pas connexe implique sphère} ]
First, assume that $\mathcal{W}(u)$ is connected. Let $\Omega$ be a connected component of $\D \backslash \mathcal{W}(u)$. It is simply connected, hence $S_\Omega$ is a disk. Keeping the notation of \ref{Theo:decomposition générale en courbes simples} in mind, let $g_{\Omega}'$ be the genus of $S_\Omega'$. The Riemann-Hurwitz formula applied to the cover $p_\Omega$ yields :
\[ 1 = \deg \left (p_\Omega \right ) \left (1- g_\Omega' \right ) - m  \]
with $m \geqslant 0$ an integer. From $m + 1 > 0$ and $1 - g_\Omega' \leqslant 1$, we deduce $g_\Omega' = 0$.

Now assume that $\mathcal{W}(u)$ is not connected. Therefore, there is a simple $J$-holomorphic curve $v : \C P^1 \to M$	such that $\Im(u) = \Im(v)$ and $\mathcal{R}^u_v \left ( \C P^1 \right ) = \D$. Notice that if $z \in \D$, $z_1 \in \C P^1$ and $z_2 \in \C P^1$ are such that $z \mathcal{R}^u_v z_1$ and $z \mathcal{R}^u_v z_2$, then $z_1 = z_2$ (in other words every element of $\D$ lifts to a unique point in $\C P^1$). Indeed by transitivity (Proposition \ref{prop:transitivityoftherelation}), we get $z_1 \mathcal{R}_v^v z_2$ and since $v$ is simple $z_1 = z_2$ (see Proposition \ref{prop:Characterizationofsimplecurves}). 

The points of $\mathcal{C}(u,v)$ are isolated. Therefore, Proposition \ref{prop:lift2} implies that the boundary of $u$ lifts to a continuous curve $\gamma : \partial \D \to \C P^1$ such that $u(z) \mathcal{R} \gamma(z)$ for $z \in \partial D$ and whose image is $\mathcal{W}(v,u)$. Hence $\mathcal{W}(v,u)$ is connected, so each connected component $\Omega$ of $\C P^1 \backslash \mathcal{W}(v,u)$ is simply connected and gives rise to a simple $J$-holomorphic disk $v_{\vert \Omega} : \Omega \to M$.

Consider $\Omega$ a connected component of $\D \backslash \mathcal{W}(u)$, and $S_\Omega$, $S_\Omega'$ as in the proof of \ref{Theo:decomposition générale en courbes simples}. If $z$ is in the interior of $S_\Omega'$, there is a point $\tilde{z} \in \C P^1$ such that $z \mathcal{R}_{v_\Omega'}^v \tilde{z}$. Let $\tilde{\Omega}$ be the connected component of $\C P^1 \backslash \mathcal{W}(v,u)$ containing $\tilde{z}$.

 Then one checks that $\mathcal{R}(\partial \tilde{\Omega} ) = \partial \Omega$ and $\mathcal{R}(\partial \Omega) = \partial \tilde{\Omega}$, so $v_\Omega'$ is a $J$-holomorphic disk by Theorem \ref{Theo:Relatively simple frames implies simple curves are equal}.
 \end{proof}

We will give the proof of Propositon \ref{prop: frame pas connexe implique sphère} at the end of the next subsection after some preliminary results.

 \subsubsection{Cutpoints and holomorphic spheres}
 
 Here we will state some results whose proofs are given in \cite{La11}. An exception is point $(2)$ of Proposition \ref{prop:premièrepropdescutpoints} which is specific to our own situation. Nevertheless, for the convenience of the reader, we shall sum up the main arguments.
 
 First, we will need to define the notion of \emph{cutpoint}, those are the points at the boundary where the disk closes on itself.
 
 \begin{defi}
 \label{defi:cutpoint}
 Let $u : (S, \partial S) \to (M,i(L))$ be a finite-energy $J$-holomorphic curve with corners and $z \in \partial S$.
 
 The point $z$ is a \emph{cutpoint} if there is a complex embedding $h : (\D^+, \partial \D^+) \to (S,\partial S)$ with $h(0) = z$ and a $J$-holomorphic disk such that	 $0$ is a dead-end of $\mathcal{W}(v, u \circ h)$.
 
 We denote by $\Cut(u) \subset \partial S$ the set of cutpoints of $u$.
 \end{defi}
 
Here are some properties of the set of cutpoints.

 \begin{prop}
 \label{prop:premièrepropdescutpoints}
 Let $u : (S, \partial S) \to (M,i(L))$ be a finite-energy $J$-holomorphic curve with corners.
 	\begin{enumerate}
 		\item If $\Cut(u) = \emptyset$, then $\mathcal{W}(u)$ has no dead-ends.
 		\item If $z \in \Cut(u)$, then $z$ is not a corner point.
 		\item If $z \in \Cut(u)$, there is a neighborhood $\omega$ of $z$ in $\partial S$ and a continuous involution $\sigma: \omega \to \omega$ such that $\sigma(z) = z$ and $z \mathcal{R}_u^u \sigma(z)$ for $z \in \omega$.	
 	\end{enumerate}

\end{prop}

\begin{proof}
 The proof of $(1)$ is clear. Assume $z_0$ is a dead-end, then there is a point $z \in \partial S$ such that $z_0 \mathcal{R}_u^u z_0$. Choose an embedding $\phi : \D \to S$ such that $\phi(0) =  z_0$ and $\phi^{-1}(\mathcal{W}(u))$ is an embedded Jordan arc. By definition $\mathcal{W}(u \circ \phi, u ) = \phi^{-1}(\mathcal{W}(u))$, so $0$ is a dead end of $\mathcal{W}(u \circ \phi, u )$. Now choose an embedding $h :(\D^+, \partial \D^+) \to (S, \partial S)$ with $h(0) = z$. Then $\mathcal{W}(u \circ \phi, u \circ h) \subset \mathcal{W}(u \circ \phi,u)$ and the former is open in the latter by Proposition \ref{prop:Theframerelationisopen}. Hence $0$ is a dead-end of $\mathcal{W}(u \circ \phi, u \circ h )$ and $z \in \Cut(u)$. 

Let $z$ be a cutpoint. Assume by contradiction that $z$ is a corner point mapping to $x = i(p) = i(1)$. There are a disk $v : \D \to M$ and an embedding $h : (\D^+, \partial \D^+) \to (S,\partial S)$ with $h(0) = z$ such that $0 \mathcal{R}_v^{u \circ h} z$ and $0$ is a dead end of $W(v,u \circ h)$. Without loss of generality, we can assume that $u\circ h \left (\R_+ \right ) \subset L_p$ and $u \circ h \left (\R_- \right ) \subset L_q$. The paths $\gamma_+ : [0,1) \to \D^+$ and $\gamma_- : (-1,0] \to \D^+$ defined by $\gamma_{\pm}(t) = t$ lift to continuous paths $\tilde{\gamma}_{\pm}$ with values in $\mathcal{W}(v, u \circ h)$ such that $\tilde{\gamma}_{\pm}(t) \mathcal{R}_v^{u \circ h} \gamma_{\pm}(t) $. Since the image of $\gamma_{\pm}$ is not contained in $\mathcal{C}(u \circ h, v)$, the paths $\tilde{\gamma}_{\pm}$ are not constant. Hence, since the frame is locally path-connected, there is a small neighborhood $\omega$ of $0$ in $\mathcal{W}(v, u \circ h)$ such that $v(\omega) \subset L_p$ and $v(\omega) \subset L_q$, so $v(\omega) \subset \{0 \}$. This implies that $v$ is constant. This contradiction proves $(2)$.

As before, assume that $z \in \Cut(u)$ and keep the notations of the proof of $(2)$. By Proposition \ref{prop:local form along a branch} and $(2)$, one can assume that $h$ is such that in a suitable local chart $u \circ h(z) = A(z^k, U(z))$ with $U(z) = o(z^k)$. We conclude that the paths $\tilde{\gamma}_{\pm}$ are embeddings with values in $\mathcal{W}(v, u \circ h)$ which is one-dimensional. Hence $v \circ \gamma_+(t) = v \circ \gamma_-(t)$ (and $k$ is even). The involution $\sigma$ maps $h \left ( \gamma_+(t) \right )$ to $h \left (\gamma_-(t) \right )$.	

\end{proof} 

The next proposition gives a sufficient condition for a holomorphic disk to be a sphere.

\begin{prop}
\label{prop:premiere condition sous laquelle une courbe est une sphere}
Let $u$ be a $J$-holomorphic disk. Assume $\mathcal{W}(u)$ is open in $\partial S$. If there is a $J$-holomorphic disk $v: \D \to M$ such that $\mathcal{W}(v,u)$ is an embedded Jordan curve, then there is a simple $J$-holomorphic sphere $w : \C P^1 \to M$ such that
	\[ \Im(u) = \Im(w) , \]
	and $\mathcal{R}_u^w (\C P^1) \neq \emptyset$.
\end{prop}

\begin{proof}
	The idea of the proof is that the boundary of the disk $u$ closes itself on the image of $\mathcal{W}(v,u)$ by $v$.

First, one can assume without loss of generality that $v$ is a simple disk.

Let $z_0$ be an extremity of $\mathcal{W}(v,u)$ and choose a point $z \in \partial \D$ such that $z_0 \mathcal{R}_v^u z$ (in particular $z_0$ is a cutpoint of $u$). One can prove as in the preceding proposition that there are two distinct paths $\gamma_{\pm} : \R_+ \to \partial \D$, and a path $\tilde{\gamma}: \R_+ \to  \mathcal{W}(v,u)$ satisfying
\begin{enumerate}
	\item $\gamma_\pm(0) = z$,
	\item $\tilde{\gamma} (0) = z_0$,
	\item $\gamma_\pm(t) \mathcal{R}_u^v \tilde{\gamma}(t) $. 	
\end{enumerate}
	
Let $N > 0$ be the first number such that $\gamma_+(t) = \gamma_-(t)$. The surface $ S:= \D / \sim $ with $\gamma_+(t) \sim \gamma_=(t)$ for $t \in [0,N]$ is a topological sphere. The map $u$ factors through the quotient projection $\pi : \D \to S$ to give a map $w : S \to M $. It remains to see that $S$ admits a structure of Riemann surface such that $\pi$ is holomorphic.

To construct charts, consider a point $\gamma_+(t_0) \mathcal{R}_u^u \gamma_-(t_0)$ with $t_0 \in (0,N)$ 	which is not a cutpoint. The proof of Proposition \ref{prop : biholomorphisms at the boundary } shows that there are local charts $h_{\pm} : (\D^+, \partial \D^+)  \to (\D, \partial \D) $ and $\tilde{h}_\pm (\D^+, \partial \D+) \to (\D, \mathcal{W}(v,u))$ such that $h_\pm (z_1) \mathcal{R}_u^v \tilde{h}_\pm (z_2)$ if and only if $z_1 = z_2$. Since $v$ is simple, there is a unique map $\phi_-$ such that $\tilde{h}_+(t) = \tilde{h}_-(\phi_- (t))$. The surface $\D^+ \sqcup \D^+ / \sim$ where $t \sim \phi_-(t)$ has a structure of a Riemann surface with the charts given by the union of the maps $\tilde{h}_+$ and $\tilde{h}_-$. The chart for the surface $S$ is then given by $h_+ \sqcup h_-$. The map $w$ restricted to this chart is holomorphic since equal to the restriction of $v$ to the images of $\tilde{h}_+$ and $\tilde{h}_-$.

If we consider a point $\gamma_+(t_0)$ which is a cutpoint, one can check that $\tilde{\gamma}(t_0)$ is an endpoint of $\mathcal{W}(v,u)$.
\end{proof}

\begin{prop}
Let $u : (\D, \partial \D) \to (M,L)$ be a finite $J$-holomorphic disk such that $\partial \D$ is open in $\mathcal{W}(u)$. Moreover, assume that $\mathcal{W}(u)$ is not connected and that $u$ is not a $J$-holomorphic sphere.

Pick $z_0,z_1 \in \Cut(u)$ such that $z_0 \mathcal{R}_u^u z_1$ and let $\gamma : [0,1] \to \partial \D$ be an embedded path with $\gamma(0) = z_0$ and $\gamma(1) = z_1$.

There are $0 < t_0 < t_1 < 1$ such that $\gamma(t_i) \in \mathcal{C}(u,u) \backslash \Cut(u)$ with $i \in \{ 0,1 \}$. 
\end{prop}

\begin{proof}
The point $t_1$ is the smallest $t \in (0,1]$ such that $\gamma(t) \in \mathcal{C}(u,u)$. It is enough to show that $\gamma(t_1) \notin \Cut(u) $ since this implies $t_1 \neq 1$.
	
The idea is as follows. Assume that $\gamma(t_1)$ is a cutpoint. Pick a connected component $C \subset \mathcal{W}(u) \cap \Int(\D)$ and a lift $\tilde{\gamma}: [0,1] \to C$ such that $\tilde{\gamma}(t) \mathcal{R}_u^u \gamma(t)$. Since $\gamma(0) \in \Cut (u)$ (resp. $\gamma(1) \in \Cut(1)$), there is a $J$-holomorphic disk $w_0 : \Int(\D) \to M$ (resp. $w_1 : \Int(\D) \to M$) and an embedding $h_0 : (\D^+, \partial \D^+) \to (M, i(L))$ (resp. $h_1 : (\D^+, \partial \D^+) \to (M,i(L))$) such that $0$ is a dead-end of $\mathcal{W}(w_0, u \circ h_0)$ (resp. $\mathcal{W}(w_1, u \circ h_1)$). We also pick an embedding $h_2 : \Int(\D) \to \Int(\D)$ such that $h_2(-1,1) = \tilde{\gamma}(\varepsilon, t_0 - \varepsilon)$ and $0 \notin \mathcal{R}_{w_i}^{w_2} w_2(\D)$. Now we attach the three disks $w_0$, $w_1$ and $w_2$ using the relation $\mathcal{R}$ to obtain a disk $w$ such that $\mathcal{W}(w,u)$ is a Jordan arc. Therefore, $u$ is a $J$-holomorphic sphere, a contradiction.
	
Let $t_0$ be the largest $t \in (0, t_1]$ such that $\gamma(t) \in \mathcal{C}(u,u)$. We are done if we show that $t_0 \neq t_1$. 
	
Assume $t_0 = t_1$. Since $\gamma(0) \mathcal{R}_u^u \gamma(1)$, one can choose $\gamma $ such that $\gamma(t) \mathcal{R}_u^u \gamma(1-t)$. One can then check that this implies $t_0 \in \Cut(u)$, a contradiction.
\end{proof}

All of this allows us to show that some disks are equivalent to disks whose frame does not possess dead-ends.

\begin{prop}
\label{prop:undisquesansdeadends}
Let $u : (\D, \partial \D) \to (M,L)$ be a finite-energy $J$-holomorphic disk with boundary on $L$ such that $\partial S$ is open in $\mathcal{W}(u)$. The frame $\mathcal{W}(u)$ is not connected and which is not a $J$-holomorphic sphere.

There is a finite-energy $J$-holomorphic disk with corners $\tilde{u}$ such that 
\begin{enumerate}
	\item $\Im(u) = \Im  (\tilde{u}  )$,
	\item we have $\Cut(\tilde{u}) = \emptyset$,
	\item there is a surjection $\pi_0 \left (\mathcal{W}(u) \right) \to \pi_0 \left ( \mathcal{W}(\tilde{u})\right )  $.	
\end{enumerate}
\end{prop}

\begin{proof}
 The idea is to fold the boundary along the cutpoints. 
 
 More precisely, choose a a point $z_1 \in \Cut(u)$ and let $\{z_1, \ldots, z_N\}$ be the set $\mathcal{R}_u^u \{ z_1 \} \cap \partial \D$. Let $\tilde{z} \in \Int(\D)$ be a point such that $z_1 \mathcal{R}_u^u \tilde{z}$. There are injective paths $\gamma_2 : [0,1] \to \mathcal{W}(u)$ and $\gamma_{i,\pm} : [0,1] \to \partial \D$ such that
 \begin{enumerate}
 	\item $\tilde{\gamma}(t) \mathcal{R}_u^u \gamma_{i,\pm}(t)$ for $i \in \{1, \ldots, N\}$,
 	\item $\tilde{\gamma}(0) = \tilde{z}$, $\tilde{\gamma}(1) \in \mathcal{C}(u,u)$,	
 \end{enumerate}
Notice that the preceding proposition shows that the points $\gamma_{\pm,i}(1)$ are distinct for $i = 1 \ldots N$.

We let $S := \D / \sim$ be the quotient of $\D$ identifying $\gamma_{\pm,i}(t)$ with $\gamma_{\pm,j}(t)$ for $i,j \in \{ 1, \ldots, N \}$. Topologically, the surface $S$ is a disk. Then $u$ factors through a map $v : S \to M $. It remains to show that there is a complex structure on $S$ such that the quotient map $\pi : \D \to S$ is holomorphic.

As in Proposition \ref{prop:premiere condition sous laquelle une courbe est une sphere}, the idea is to build a chart around $\gamma_{i,+}(t)$ using as a chart a quotient of a small disk around the corresponding point $\tilde{\gamma}(t)$. For $\gamma_{i,+}(t)$ with $t \in [0,1)$ it is the same process as in Proposition \ref{prop:premiere condition sous laquelle une courbe est une sphere}.

Consider the points $\gamma_{i,+}(1)$ and $\gamma_{i,-}(1)$ and assume that $\tilde{\gamma}(1)$ is not a vertex of the graph $\mathcal{W}(u)$. Choose a small enough holomorphic embedding $\phi : \Int \D \to \D$ such that $\phi(0) = \tilde{\gamma}(1)$ and $\Im \phi \cap \mathcal{C}(u,u) = \tilde{\gamma(1)}$. The set $\phi^{-1} \left ( \mathcal{W}(u,u) \right )$ is divided in two arcs, one is simply $\phi^{-1}(\tilde{\gamma})$ and we call the other $\gamma'$. Choose an embedding $h' : (\D^+, \partial \D^+) \to (\D, \gamma')$. The graph $\mathcal{W}(u \circ \phi \circ h,u)$ consists of the boundary$(-1,1)$ and an arc going from $0$ to the outer boundary of the half-disk. We call this arc $\gamma ''$.

Using the proof of Proposition \ref{prop : biholomorphisms at the boundary }, we show that there are 4 embeddings $\tilde{h}_\pm : (\D^+, \partial \D^+) \to (\D^+, \partial \D^+ \cup \gamma'') $ and $h_\pm : (\D^+, \partial \D^+) \to (\D, \partial D)$ with $\tilde{h}_\pm(0) = 0$ and $h_\pm (0) = \gamma_\pm(1)$. These satisfy $\tilde{h}_\pm (z_1) \mathcal{R} h_\pm(z_2) $ if and only if $z_1 = z_2$. As before one can attach the two half-disks along their boundaries and identify the resulting surface with $\D^+$ through the disjoint union of the maps $\tilde{h}_\pm$.

Suppose that $\tilde{\gamma}(1)$ is a vertex of the graph $\mathcal{W}(u)$. Let $\phi : \Int \D \to \D$ be a small enough holomorphic embedding so that $\phi(0) = \tilde{\gamma}(1)$ and $\Im \phi \cap \mathcal{C}(u,u) = \tilde{\gamma(1)}$. The paths $\gamma_{\pm}(1-t)$ lift to two (not necessarily distinct) arcs $\gamma_1$ and $\gamma_2$ in $\mathcal{W}(u,u)$. Let $\tilde{h}: (\D^+, \partial \D^+) \to  \left (\Int \D, \phi^{-1}(\gamma_1 \cup \gamma_2) \right )$ be a holomorphic embedding with $\phi^{-1}(\tilde{\gamma}) \subset \Im(\tilde{h})$. We then proceed just as before!

The end product is a finite-energy $J$-holomorphic disk with corners $v$ such that $\# \Cut(v) \leqslant \# \Cut(u) -1$. We then repeat the process by induction to get the desired disk $\tilde{u}$.  
 \end{proof}
 
 After these results, we now return to the proof of Proposition \ref{prop: frame pas connexe implique sphère}.
  
 \begin{proof}[Proof of Proposition \ref{prop: frame pas connexe implique sphère}]
 	Assume by contradiction that $u$ is not a $J$-holomorphic sphere. Then by Propositions (\ref{prop:undisquesansdeadends}), there is a finite-energy $J$-holomorphic disk $\tilde{u}$ with corners and boundary on $L$ which satisfies the following.
 	\begin{enumerate}
 		\item The set $\partial \D$ is open in $\mathcal{W}(u)$.
 		\item The connected component $\Omega$ of $\D \mathcal{W}(\tilde{u})$ which contains $\partial \D$ is not simply connected. It is the unique connected connected component with this property.
 		\item The set of cutpoints is empty.
 		\item The map $\tilde{u}$ is not a $J$-holomorphic sphere.		
 	\end{enumerate}
 Call $\Omega_1$ the connected component of $\D \backslash \mathcal{W}(u)$ which contains $\partial \D$ and choose a map $h :(S, \partial S) \to (\Omega_1, \partial \Omega_1)$ which is a biholomorphism from $\Int(S)$ to $\Int(\Omega_1)$. The map $u_{\Omega_1} := u \circ h$ factors through a simple $J$-holomorphic map $v_{\Omega_1} : S_{\Omega_1}' \to M$. We show that this is a disk. 
 
 For $C$ a connected component of $\mathcal{W}(u)$, denote by $\Omega_C$ the connected component of $\D \backslash \mathcal{W}(u)$ with boundary $C$. It is simply connected, hence biholomorphic to a disk. Hence, the map $u_{\Omega_C} : = u_{\vert \Omega_C}$ factors through a simple disk $v_{\Omega_C}$.
 
 Let $z \in \partial S_{\Omega_1}'$ and pick a point $\tilde{z} \in p_{\Omega_1}^{-1}(z)$. There is $s \in C$ such that $h(\tilde{z}) \mathcal{R}_u^u s$. Then, either $h(\tilde{z}) \mathcal{R}_{u_{\Omega_1}}^{u_{\Omega_C}} s$ or $h(\tilde{z}) \mathcal{R}_{u_{\Omega_1}}^{u_{\Omega_1}} s$. In the first case, the $J$-holmorphic maps $v_{\Omega_1}$ and $v_{\Omega_C}$ satisfy $\mathcal{W}(v_{\Omega_1}, v_{\Omega_C}) = \partial \Omega_1 $ and $\mathcal{W}(v_{\Omega_C},v_{\Omega_1}) = \partial \Omega_C$. Hence, $v_{\Omega_1}$ and $v_{\Omega_C}$ are conjugate.
 
 If there is no connected component such that $h(\tilde{z}) \mathcal{R}_{u_{\Omega_1}}^{u_{\Omega_C}} s$, the surface $S_{\Omega_1} / \mathcal{R}_{u_{\Omega_1}}^{u_{\Omega_1}}$ has a unique connected component. Therefore, it is a disk
 
 Now choose a connected component $C$. We glue the disks $v_{\Omega_C}$ and $v_{\Omega}$ along their boundaries to get a $J$-holomorphic sphere $v : \C P^1 \to M $ such $\mathcal{R}_{\tilde{u}}^v(\C P^1) \neq \emptyset$. We readily conclude that $\tilde{u}$ is a sphere. This is a contradiction.
\end{proof}

\section{Consequences of the main theorem }

\subsection{Simplicity of curves for generic almost complex structures}
In this subsection, we give the proof of Corollary \ref{coro:Jholdisksaremultiplycovered}. The proof is basically contained in \cite[Theorem B]{La11} and \cite{BC07}. Here, we sum up the main arguments involved in the proof.

\subsubsection{Intersection points and indices of curves.} 
\label{subsubsection:Interstionpointsindices}

For each (ordered) double point $(p,q) \in R$ (with as usual $x = i(p) =i(q)$), denote by $\mathcal{G}(T_x M)$ the Lagrangian Grassmannian of $T_x M$. We choose once and for all a smooth path $\lambda_{(p,q)} : [0,1] \to \mathcal{G}(T_x M)$ such that $\lambda_{(p,q)}(0) = d i_p (T_x L)$ and $\lambda_{(p,q)}(1) = d i_q (T_x L)$. Moreover, we may assume that $\lambda_{(q,p)}$ is $\lambda_{(p,q)}$ parameterized in the reverse direction.

Now define a Maslov pair $(E,F)$ (we use the terminology of \cite[Appendix C.3]{McDS12}) as follows. We let $E$ be the trivial symplectic vector bundle over the closed Poincaré half-plane $\HP$ with fiber $T_x M $ equipped with the symplectic form $\omega_x$. Now consider a strictly increasing smooth function $f : \R \to [0,1]$ such that $f(t) = 0 $ for $t << 0$ and $f(t)  = 1$ for $t >> 0$. Then the Lagrangian boundary condition is given for $t \in \R$ by $F_t = \lambda_{(p,q)}(f(t))$.

We endow $\HP$ with the following strip-like end
\[\varepsilon : \begin{array}{ccc} ] - \infty, 0 ] \times [0,1] & \to & \HP \\ (s,t) & \mapsto & e^{- \pi (s+ it)} \end{array}  \]
and endow the bundle $\HP \times T_x M$ with the trivial symplectic connection. This satisfies the hypotheses of \cite[section 8h]{Sei08}, hence admits an associated Fredholm Cauchy-Riemann operator. We denote by $\Ind(p,q)$ the index of this operator.
 
Now, choose a compatible almost complex structure $J \in \mathcal{J}(M,\omega)$ and let 
\[(p_0,q_0), \ldots, (p_d,q_d) \in \R\]
be $d \in \N^*$ ordered self-intersection points of $i$.	

Let $A$ be a homotopy class of topological disks with corners\footnote{See Remark \ref{rk:defoftopologicalcurves} for the definition} and corner points given in cyclic order by $(p_0,q_0), \ldots , (p_d, q_d)$. 

Assume first $d \geqslant 2$. Recall that there is a universal family $\mathcal{S}^{d+1} \xrightarrow[]{\pi} \mathcal{R}^{d+1}$ of disks with $d+1$ marked points. Fix a universal choice of positive strip-like ends \footnote{See \cite[Section (9)]{Sei08} for the relevant definitions}. 

We define $\mathcal{M}(A, (p_0,q_0), \ldots, (p_d,q_d), J)$ to be the space of maps $u : \pi^{-1}(r) \to M $ for some $r \in \mathcal{R}^{d+1}$ satisfying the following conditions,
\begin{enumerate}
	\item $u$ is a finite-energy $J$-holomorphic disk with corners and boundary on $L$,
	\item the corner points of $u$ coincide with the limits of the strip-like ends and the switch condition at the $i$-th marked point is given by $(p_i,q_i)$,
	\item the homotopy class of $u$ is $A$.
\end{enumerate}

Each $u \in \mathcal{M}(A, (p_0,q_0), \ldots, (p_d,q_d), J)$ gives rise to the bundle pair $(u^{*} TM , u^{*} TL )$. The linearization of the Cauchy-Riemann equation at $u$ yields a Cauchy-Riemann operator between suitable Sobolev completions of the space of sections of this bundle pair
\[ D_u : W^{k,p} (u^{*} TM , u^{*} TL ) \to W^{k-1,p}(\Lambda^{0,1} u^{*} TM). \]

Fix such a $u : \pi^{-1}(r) \to M$ and denote by $x_0, \ldots , x_d$ the marked points in the domain. There is a natural compactification of $\pi^{-1}(r)$ given by the union of $\pi^{-1}(r)$ and $d+1$ copies of the interval $[0,1]$ topologized so that the positive (resp. negative) strip-like ends $\varepsilon_i : (0,+\infty[ \times [0,1] \to \pi^{-1}(r)$ (resp. $\varepsilon_i : ]-\infty,0) \times [0,1] \to \pi^{-1}(r)$) extend to homeomorphisms $\varepsilon_i : (0,+\infty] \times [0,1] \to \pi^{-1}(r)$ (resp. $\varepsilon_i : [-\infty, 0 ) \times [0,1] \to \pi^{-1}(r)$). We will call it $\overline{\pi^{-1}(r)}$.

The map $u$ admits a unique extension to a continuous map $\overline{u} : \overline{\pi^{-1}(r)} \to M$. This gives rise to a Maslov pair by setting the boundary condition to be $u^*TL$ over $\partial \pi^{-1}(r)$ and $\lambda_{(p,q)}$ over the added intervals. The index of this Maslov pair only depends on the homotopy class $A$ of the map $u$. We call it $\mu_A$.
 
The index of the operator $D_u$ is given by the following formula 
	\[ \Ind(D_u) = n + \mu_A - \sum_{i=0}^d \Ind(p_i,q_i). \]
The reader may find a proof in the paper of Akaho-Joyce \cite[Section 4.3, Proposition 4.6]{AJ10}. One can also deduce it from the exposition in Seidel's book \cite[Section (11)]{Sei08}.

For the case $ d = 1$, we consider the space of $J$-holomorphic strips with corners at $(p_0,q_0)$ and $(p_1,q_1)$. 	More precisely, define $Z = \R \times [0,1]$. We denote by $\tilde{\mathcal{M}}(A, (p_0,q_0), (p_1,q_1), J)$ the space of finite-energy $J$-holomorphic maps $u : Z \to M$ such that $u(0, \cdot)$ (resp. $u(1, \cdot)$) lifts to a map $\gamma_- : \R \to L$ (resp. $\gamma_+ : \R \to L$) with $\lim_{s \to + \infty} (\gamma_-(s), \gamma_+(s) ) = (p_1,q_1)$, $\lim_{s \to - \infty} (\gamma_-(s), \gamma_+(s)) = (p_0,q_0)$.
The index of such a curve is given by
\[ \Ind(D_u) =  \mu_A + \Ind(p_0,q_0) - \Ind(p_1,q_1) . \]

For the case $d= 0$, we consider the space of $J$-holomorphic teardrops with corner at $(p_0,q_0) $. More precisely, we denote by $\tilde{\mathcal{M}}(A, (p_0,q_0), J)$ the space of finite-energy $J$-holomorphic maps $u : \HP \to M$ such that $u_{\lvert \R}$ lifts to a map $\gamma : \R \to L$ with $\lim_{s \to - \infty} \gamma(s) = p_0$ and $\lim_{s \to + \infty} \gamma(s) = q_0$.

The index of an element $u$ of this moduli space is given by
\[ \Ind(D_u) = \mu_A + \Ind(p_0,q_0) . \]

\subsubsection{Generic transversality of moduli spaces}	
Assume that $d \geqslant 2$ and denote by $\mathcal{M}^*(A, (p_0,q_0), \ldots, (p_d,q_d), J) \subset \mathcal{M}(A, (p_0,q_0), \ldots, (p_d,q_d), J)$ the moduli space of simple curves with corners the $(p_i,q_i)$. The usual transversality arguments imply that there is a second category subset $\mathcal{J}_{reg}(M, \omega) \subset \mathcal{J}(M, \omega)$ such that for $J \in \mathcal{J}_{reg}(M, \omega)$ the space $\mathcal{M}^*(A, (p_0,q_0), \ldots, (p_d,q_d), J)$ is a manifold of dimension $\Ind(D_u) + d-2$ if not empty.

If $d \in \{0,1\}$, we quotient $\tilde{\mathcal{M}}^*(A, (p_0,q_0), \ldots, (p_d,q_d), J)$ by the space of conformal reparameterizations leaving the marked points fixed and denote the resulting space by $\mathcal{M}^*(A, (p_0,q_0), \ldots, (p_d,q_d), J)$. There is a second category subset $\mathcal{J}_{reg}(M, \omega) \subset \mathcal{J}(M, \omega)$ such that the space $\mathcal{M}^*(A, (p_0,q_0), \ldots, (p_d,q_d), J)$ is a manifold of dimension $\Ind(D_u) + d-2 $ if not empty.

Assume now that $A_1$ and $A_2$ are two homotopy classes of topological disks with corners at $(p_0,q_0), \ldots , (p_d,q_d)$ and $(\tilde{p}_0,\tilde{q}_0), \ldots , (\tilde{p}_m,\tilde{q}_m)$ respectively. We define 
\[\mathcal{M}^*(A_1,A_2,(p_0,q_0), \ldots , (p_d,q_d),(\tilde{p}_0,\tilde{q}_0), \ldots , (\tilde{p}_m,\tilde{q}_m),J  )\]
to be the set of pairs of simple disks $(u_1,u_2)$ such that $u_1(\D) \not \subset u_2(\D)$ and $u_2(\D) \not \subset u_1(\D)$. There is a second category subset $\mathcal{J}_{\text{reg}}(M,\omega)$
 such that for $J \in \mathcal{J}_{reg}(M,\omega)$ the space $\mathcal{M}^*(A_1,A_2,(p_0,q_0), \ldots (\tilde{p}_m,\tilde{q}_m),J  )$ is a smooth manifold of dimension $2n + \mu_{A_1} + \mu_{A_2} - \sum_i \Ind(p_i,q_i) - \sum_i \Ind(\tilde{p}_i, \tilde{q}_i) $.
 
 Now for $k \geqslant 0$ consider the moduli space of (parameterized) pairs of disks with marked points at the boundary
 \[ \mathcal{M}_k^*(A_1,A_2,(p_0,q_0), \ldots (\tilde{p}_m,\tilde{q}_m),J  ) : = \mathcal{M}^*(A_1,A_2,(p_0,q_0), \ldots (\tilde{p}_m,\tilde{q}_m),J  ) \times (\partial D)^{2k} . \]
 There is a smooth evaluation map
 \[  \ev_k : \begin{array}{ccc}\mathcal{M}_k^*(A_1,A_2,(p_0,q_0), \ldots (\tilde{p}_m,\tilde{q}_m),J  ) & \to & L^{2k} \\
 (u_1, u_2, z_1, \ldots z_k, x_1, \ldots , x_k) & \mapsto	 (u_1(z_1), u_2(x_1), \ldots , u_1(z_k), u_2(z_k)) 
 \end{array}. \]
 
Denote by $\Delta = \ens{(x,x)}{x \in L} \subset L \times L$ the diagonal. There is a second category subset $\mathcal{J}_{\text{reg}}(M, \omega)$ such that for every $J \in \mathcal{J}_{\text{reg}}(M, \omega)$ and $k \geqslant 1$, the evaluation map $\ev_k$ is transversal to the product $\Delta^k$. Hence, if not empty, the set $\ev_k^{-1}(\Delta^k)$ has the structure of a smooth manifold of dimension $2n + \mu_{A_1} +\mu_{A_2} - \sum \Ind(p_i,q_i) - \sum \Ind(\tilde{p}_i,\tilde{q_i})+ (2-n)k$.

Assume that $n \geqslant 3$, then for $k$ large enough, we have $2n + \mu_{A_1} +\mu_{A_2} - \sum \Ind(p_i,q_i) - \sum \Ind(\tilde{p}_i,\tilde{q_i})+ (2-n)k \leqslant 0$, so $\ev_k^{-1}(\Delta^k)$ is empty. We conclude that the following proposition holds.

\begin{prop}
	There is a second category subset $\mathcal{J}_{\text{reg}}(M,L,\omega) \subset \mathcal{J}(M,\omega)$ and a number $k_0$ such that if $J \in \mathcal{J}_{\text{reg}}(M,L,\omega)$, $u_1 \in \mathcal{M}^*(A_1, (p_0,q_0), \ldots, (p_d,q_d), J) $ and $u_2 \in \mathcal{M}^*(A_2, (\tilde{p}_0,\tilde{q}_0), \ldots, (\tilde{p}_m,\tilde{q}_m), J)$ satisfy $u_1(\D) \not \subset u_2(\D)$ and $u_2(\D) \not \subset u_1(\D)$, then the number of $z_1, z_2 \in \partial \D$ such that $u_1(z_1) = u_2(z_2)$ is finite. 
\end{prop}

The same argument for self intersections yields
\begin{prop}
There is a second category subset $\mathcal{J}_{\text{reg}}(M, L, \omega) \subset \mathcal{J}(M,\omega)$ such that if $u \in \mathcal{M}^*(A, (p_0,q_0), \ldots, (p_d,q_d), J)$, then the number of pairs $(z_1, z_2) \in \partial D \times \partial D$ such that $u(z_1) = u(z_2)$ is finite.
\end{prop}

The rest of the proof follows from \cite[Theorem B]{La11}.

\subsubsection{Generically teardrops and strips are simple}

A finite-energy $J$-holomorphic disk with corners and boundary on $L$ is a \emph{teardrop} if it has a unique corner point. An elementary argument gives the following consequence of Corollary \ref{coro:Jholdisksaremultiplycovered}.

\begin{coro}
\label{coro:generic strips are simple}
	Assume $n \geqslant 3$. There is a second category subset $\mathcal{J}_{\text{reg}}(M,L,\omega) \subset \mathcal{J}(M,\omega)$ such that the following holds. If $ J \in \mathcal{J}_{\text{reg}}(M,L,\omega)$, every finite-energy $J$-holomorphic teardrop with boundary on $L$ is simple.	
\end{coro}

\begin{proof}
	We assume that $\mathcal{J}_{\text{reg}}(M,L,\omega)$ is the second category subset given in Corollary \ref{coro: multiply covered disks} (such that all the $J$-holomorphic curves are either simple or multiply covered for $J \in \mathcal{J}_{\text{reg}}(M,L,\omega)$).
	
	Assume that $u : (\D, \partial \D) \to (M, i(L))$ is a teardrop with finite energy. Denote by $z_1$ its corner point. There is a simple disk $v$ as well as a branched cover $p : (\D, \partial \D) \to (\D, \partial \D)$ which restricts to a cover on the boundary such that $u = v \circ p$. Notice that $v$ has a corner point at $p(z_1)$. Now if the the degree of $p$ is more than $2$, we see that $u$ has corner points at $p^{-1} \{ z_1 \}$. This set is of cardinality greater or equal to 2, a contradiction.	
\end{proof}

Now, let $i_1 : L_1 \to M$ and $i_2 : L_2 \to M $ be two Lagrangian immersions. These naturally give rise to a Lagrangian immersion $i : L_1 \sqcup L_2 \to M$ which we will assume to be generic (transverse double points and no triple points). 

Let $J \in \mathcal{J}(M,\omega)$ be a compatible almost complex structure, $x_-$ and $x_+$ be two intersection points between $L_1$ and $L_2$. As usual, a $J$-holomorphic strip between $L_1$ and $L_2$ from $x_-$ to $x_+$ is a $J$ holomorphic map $u : \R \times [0,1] \to M$ such that
\begin{enumerate}
 \item $\lim_{s \to +\infty } u(s,t) = x_- , \lim_{s \to + \infty} = x_+ $,
 \item $u(\cdot, 0 ) $ (resp. $u(\cdot, 1)$ admits a continuous lift $\gamma_1$ to $L_1$ (resp. a continuous lift $\gamma_2$ to $L_2$).
\end{enumerate}
In particular, if we precompose a $J$-holomorphic strip $u$ with a biholomorphism $\psi : \D \backslash \{ -1, 1 \} \to \R \times [0,1]$, the resulting map $u \circ \psi$ is a $J$-holomorphic disk with two corner points. The boundary lifts are given by $\tilde{\gamma}_1 = \gamma_1 \circ \psi$ and $\tilde{\gamma}_2 = \gamma_2 \circ \psi$. We now have the following proposition about the structure of such strips.

\begin{prop}
\label{prop:strips are simple}
There is a second category subset $\mathcal{J}(M, L_1, L_2, \omega) \subset \mathcal{J}(M,\omega)$ such that the following holds. If $J \in \mathcal{J}(M, L_1, L_2, \omega)$, every finite-energy $J$-holomorphic strip is simple.	
\end{prop}

\begin{proof}
There is a second category subset $\mathcal{J}(M, L_1, L_2, \omega)$ such that for every $J \in \mathcal{J}(M, L_1, L_2, \omega)$, every $J$-holomorphic disk with boundary on $i : L_1 \sqcup L_2 \to M$ is multiply covered. Let us fix one such $J$.

Assume that $u : \R \times [0,1] \to M$ is a finite-energy $J$-holomorphic strip from $x_-$ to $x_+$ and $\psi : \D \backslash \{ -1,1 \} \to \R \times [0,1] $ a biholomorphism. Then $u \circ \psi$ is a finite-energy $J$-holomorphic strip with two corner points at $-1$ and $1$. It is therefore multiply covered. There is $v :\D \to M$ simple and $p : \D \to \D$ a covering such that $u = v \circ p$. Assume by contradiction that $p$ has degree more than $2$.

We claim that $-1 \notin p^{-1} (p(1))$. Indeed, for each $z \in p^{-1}(p(1))$, any conformal embedding $h : (\D^+, \partial \D^+_\R) \to (\D, \partial \D)$ such that $h(0) = z$ satisifies $u \circ h(\R^+) \subset L_2$ and $u \circ h(\R^-) \subset L_1$. For instance, this can be seen using local charts. However, a conformal local chart $h : (\D^+, \D^+_\R) \to (\D, \partial \D)$ such that $h(0) = -1$ satisfies $u \circ h(\R^-) \subset L_2$ and $u \circ h(\R^+) \subset L_1$.
\end{proof}

Now, we will need one last statement. For this, assume $d \geqslant 2$, and that $L_0, \ldots, L_d$ are $d+1$ embedded Lagrangian submanifolds in general position. For $i \in \{ 0, \ldots, d \}$, fix $x_i \in L_i \cap L_{i+1}$. A $J$-holomorphic polygon with boundary condition $L_1, \ldots, L_d $ is a $J$-holomorphic map $u : \pi^{-1}(r)  \to M $ with $r \in \mathcal{R}^{d+1}$ such that $\lim_{s \to + \infty} u \circ \varepsilon_i(s,t) = x_i$ and the image by $u$ of the arc between the $i$ and $i+1$ end is included in $L_i$.

\begin{prop}
	There is a second category subset $\mathcal{J}(M,L_0, \ldots, L_d ,\omega) \subset \mathcal{J}(M,\omega)$ satisfying the following property. If $J \in \mathcal{J}(M,L_0, \ldots, L_d ,\omega)$, every $J$-holomorphic polygon of finite energy with boundary condition $L_0, \ldots , L_d$ is simple.	
\end{prop}

\begin{proof}
As before, call $i : \bigsqcup L_i \to M$ the natural Lagrangian immersion. There is a second category subset $\mathcal{J}(M,L_0, \ldots, L_d ,\omega) \subset \mathcal{J}(M,\omega)$ such that any finite-energy $J$-holomorphic disk with boundary on $i$ is multiply covered. Fix an almost complex structure $J \in \mathcal{J}(M,L_0, \ldots, L_d ,\omega)$. 

Now let $u$ be a $J$-holomorphic polygon of finite energy and fix a biholomorphism $\psi : \D \to \pi^{-1} (r)$. Call $y_0, \ldots , y_d$ the preimages of the marked points of $\pi^{-1}(r)$ by $\psi$. Then $u \circ \psi$ is a finite-energy disk with boundary on $i$, hence is multiply covered. So there is $v: (\D, \partial \D) \to (M,i(L))$ simple and $p : (\D, \partial \D) \to (\D, \partial \D)$ a covering such that $u \circ \psi = v \circ p$.

 As before, we see that for $i \neq j$, $y_i \notin p^{-1}(p(y_j))$ since the image of any neighborhood of $z \in p^{-1}(p(y_j))$ in $\partial \D$ by $u$ intersects both $L_j$ and $L_{j+1}$. Hence, the cover $p$ is of degree $1$.
\end{proof}

\begin{rk}
	One should be aware that the set $\mathcal{J}(M, L_1, \ldots, L_d, \omega)$ depends on the submanifolds $L_0, \ldots , L_d$. In particular, the author does not know if generically any $J$-holomorphic polygon (without restriction on the boundaries) is simple. 	
\end{rk}

\subsection{Time-independent Floer homology}
As an easy application of the Main theorem, we show that Floer homology can be defined with time-independent complex structures. This is (as far as the author knows) new.

To do this, let us now assume that $(M,\omega)$ is closed and monotone. Let $[\omega] : H_2(M) \to \R $ be the morphism induced by symplectic area. There is $\lambda > 0$ such that $[\omega] = \lambda c_1(TM)$.

Let $L_1$ and $L_2$ be two embedded compact Lagrangian submanifolds and denote by $N_1 \geqslant 1$ and $N_2 \geqslant 1$ their minimal Maslov number. We assume that $N_p \geqslant 3$ for $p \in \{ 1,2\}$.

For $x$ and $y$ two distinct intersection points in $L_1 \cap L_2$, and $A$ a homotopy class of finite-energy strips from $x$ to $y$, we denote  
\begin{itemize}[label = $\bullet$]
\item by $\tilde{\mathcal{M}}(x,y,L_1,L_2,A,J)$ the set of $J$-holomorphic strips from $x$ to $y$ in the homotopy class $A$,  
\item by $\mathcal{M}(x,y,L_1,L_2,A,J)$ its quotient by the natural $\R$-action,
\item by $\tilde{\mathcal{M}}^*(x,y,A,L_1,L_2,J) \subset \tilde{\mathcal{M}}(x,y,A,L_1,L_2,J) $ the set of simple $J$-holomorphic strips,
\item by $\mathcal{M}^*(x,y,A,L_1,L_2,J)$ its quotient by the $\R$-action. 
\end{itemize}

From Proposition \ref{prop:strips are simple} and the standard transversality arguments (cf \cite{FHS95}), we immediately deduce the following.

\begin{prop}
	There is a second category subset $\mathcal{J}_{\text{reg}} (M,L_1,L_2,\omega) \subset \mathcal{J}(M,\omega) $ such that 
	\begin{enumerate}
	\item If $J \in\mathcal{J}_{\text{reg}} (M,L_1,L_2,\omega) $, for each homotopy class $A$, all $J$-holomorphic strips are simple :
	\[ \forall x,y \in L_1 \cap L_2, \mathcal{M}^*(x,y,L_1,L_2,A,J) = \mathcal{M}(x,y,L_1,L_2,A,J) . \]
	\item If $J \in \mathcal{J}_{\text{reg}} (M,L_1,L_2,\omega)$, then $\mathcal{M}(x,y,L_1,L_2,A,J)$ is a finite dimensional manifold for all $x,y \in L_1 \cap L_2 $. 
	\end{enumerate}	
\end{prop}

Now for such a generic $J$ and $k \in \N$, denote by $\mathcal{M}(x,y,L_1,L_2,A,J)^k$ the $k$ dimensional component of $\mathcal{M}(x,y,L_1,L_2,A,J)$. 

By a standard Gromov compactness argument, the set $\mathcal{M}^0(x,y,A,J)$ is compact. Furthermore, there is a 1-dimensional manifold with boundary $\overline{\mathcal{M}}^1(x,y,A,J)$ whose interior is identified with $\mathcal{M}^1(x,y,A,J)$ and whose boundary is identified with
	\[ \bigsqcup_{z \in L_1 \cap L_2} \mathcal{M}^0(x,z,A,J) \times \mathcal{M}^0(z,y,A,J) .\]
	
Consider the Novikov ring of formal power series with coefficients in $\Z_2$ :
	\[ \Lambda_{\Z_2} = \ens{\sum_{\lambda_i \to + \infty, \lambda_i \geqslant 0}a_i T^{\lambda_i}}{ a_i \in \Z_2} .\]
As usual we define the Floer complex between $L_1$ and $L_2$ to be the $\Lambda_{\Z_2}$-module generated by the intersection points
\[ CF(L_1,L_2,J) = \bigoplus_{x \in L_1 \cap L_2} \Z_2 \cdot x. \]
The differential on this complex is given by a count of rigid $J$-holomorphic strips modulo 2
\[ \mathrm{d}: \begin{array}{ccc} 
 	CF(L_1,L_2,J) & \to & CF(L_1,L_2,J) \\
 	y & \mapsto & \sum_{y \in L_1 \cap L_2} \#_{\Z_2} \mathcal{M}^0(x,y,A,J) T^{\omega(A)} y .
 \end{array}
\]
It is immediate to see from the usual Gromov compactness argument that $\mathrm{d}^2 = 0$, so $(CF(L_1,L_2),\mathrm{d})$ is a well-defined differential complex.

Since any generic almost complex structure is, in particular, a generic time-dependent almost complex structure, the homology of this complex computes the usual Lagrangian intersection Floer homology. Hence we can conclude that the following theorem is true.

\begin{Theo}
	There is a second category subset $\mathcal{J}_{\text{reg}}(M, \omega,L_1,L_2) \subset \mathcal{J}(M,\omega)$ such that the Floer complex
	\[ (CF(L_1,L_2,J),d)  \]
	is well-defined (as a differential complex). Moreover, its homology computes the usual Lagrangian intersection Floer homology.
\end{Theo}

\subsection{Work in progress}

\subsubsection{Framework}
In this section, we will describe some expected applications of the main theorems to the study of the surgery of two immersed Lagrangian submanifolds.

Let us consider a compact exact symplectic manifold $(M,\omega)$ with Liouville form $\lambda$, convex boundary and complex dimension $n \geqslant 3$. We denote by $\widehat{M}$ its completion.

\paragraph{\bf{Gradings:}} Assume that the first Chern class of $(M,\omega)$ satisfies $2 c_1(TM) =  0$ in $H^2(M,\Z)$. This implies that the complex line bundle $\Lambda^n T^*M \otimes \Lambda^n T^*M$ is trivial. Hence, it admits a non-vanishing section $\Omega$.

For each Lagrangian subspace $L \in \mathcal{G}(T_x M)$, choose a real basis $v_1, \ldots, v_n$ of $L$ and define
	\[ \deter^2_\Omega (L) = \frac{\Omega(v_1, \ldots, v_n)}{\norm{\Omega(v_1, \ldots, v_n)}} . \]
One can check that this does not depend on the choice of $v_1, \ldots, v_n$ and therefore defines a smooth function $\deter^2_\Omega : \mathcal{G}(TM) \to \R$. 	

An \emph{exact graded Lagrangian immersion} is a tuple $(L,i,f_L, \theta_L )$ with 	
\begin{itemize}[label = $\bullet$]
	\item $L$ a compact manifold,
	\item $i : L \looparrowright M$ a generic Lagrangian immersion,
	\item $f_L : L \to \R$ a smooth function such that $i^* \theta = d f_L$,
	\item $\theta_L : L \to \R$ a smooth function such that 
	\[ e^{2i\pi \theta_L} = \deter^2_\Omega \circ \iota  	 \]
	where $\iota : L \to \mathcal{G}(TM)$ is the map $x \mapsto \Im(di_x)$.	
\end{itemize}

Now assume that $L_1$ and $L_2$ are two transverse exact graded immersions which intersect transversally. Let $x$ be an intersection point of these. Fix an adapted almost complex structure $J$ and denote by $\alpha_1, \ldots, \alpha_n$ the Kähler angles of the pair $\left (T_x L_1, T_x L_2 \right )$. The \emph{index} of $x$ as an element of $CF(L_1,L_2)$ is the number
\[  \norm{x} = n + \theta_{L_2}(x) - \theta_{L_1}(x) - \frac{\alpha_1 + \ldots + \alpha_n}{\pi}.\]

Similarly, we can define the \emph{index} of a self intersection point $(p,q) \in R$ of $L_i$ for $i = 1,2$ :
\[ \norm{(p,q)} =  n + \theta_{L_2}(q) - \theta_{L_1}(p) - \frac{\alpha_1 + \ldots + \alpha_n}{\pi}, \]
where $\alpha_1, \ldots, \alpha_n$ are the Kähler angles of the pair $ \left (d i_p (T_pL), d i_q(T_q L) \right)$ with respect to $J$.

The reader may wonder if this is the same as the index defined in the subsection \ref{subsubsection:Interstionpointsindices}. There is a choice of path $\lambda_{(p,q)}$ such that $\Ind(p,q) = \norm{(p,q)}$ : it is explained in \cite[(11g)]{Sei08}. 
 
\paragraph{\bf{Lagrangian surgery:}} Following the presentation of Biran and Cornea (\cite[6.1]{BC13}), we will describe the surgery of $L_1$ and $L_2$ at an intersection point. 

First, for each intersection point $y$ between $L_1$ and $L_2$, fix a Darboux chart $\phi_y : B(0,r_y) \to (M,\omega)$ such that $\phi_y(0) = y$, $\phi_y(\R^n) \subset L_1$, $\phi_y(i \R^n) \subset L_2$ and whose image does not contain any other intersection point.

Now consider a smooth path $\gamma(t) : = (a(t), b(t)) \in \C$, with $t \in \R$, such that $\gamma(t) = (t,0)$ for $t < -1$, $\gamma(t) = (0,t)$ for $t > 1$ and $a'(t) , b'(t) > 0$ for $t \in (-1,1)$. For $\varepsilon > 0$, the set 
\[ H_\varepsilon : = \ens{(x_1 \gamma(t), \ldots, x_n \gamma(t))}{t \in \R, \ (x_1, \ldots, x_n) \in S^{n-1}} \] 
is a smooth Lagrangian submanifold of $\C^n$.

Now let $x \in L_1 \cap L_2$ be an intersection point. For $\varepsilon > 0$ small enough, there is a generic immersion $L_1 \#_{x,\varepsilon} L_2$ obtained by removing $\phi_x(\R^n \cup i\R^n)$ and replacing it by $\phi_x(H_\varepsilon)$. This has domain the connected sum of $L_1$ and $L_2$ and turns out to be exact. Moreover, if $\norm{x} = 1$, the gradings of $L_1$ and $L_2$ canonically induce a grading $\theta$ on $L_1 \#_{x,\varepsilon}L_2$ which agrees with $\theta_{L_1}$ and $\theta_{L_2}$ on $L_1 \sqcup L_2 \backslash \phi_x(B(0,r_x))$ (this is a result of Seidel \cite[Lemma 2.13]{Sei00}).

From now on, we will consider the case where $L_1$ and $L_2$ are actually embedded and we will fix an intersection point $x \in L_1 \cup L_2$ of degree 1.

Lastly, we will somewhat restrict the space of almost complex structures we consider. We denote by $\mathcal{J}_\phi(M,\omega)$ the set of adapted almost complex structures which agree with $(\phi_y)_* J_{\text{std}}$ on $\phi_y(B(0,r_y))$ for each intersection point $y$.

The proof of \ref{coro:Jholdisksaremultiplycovered} shows that there is a second category subset $\mathcal{J}_{\phi,\text{reg}}(M,\omega) \subset \mathcal{J}_\phi(M,\omega)$ such that 
\begin{enumerate}
	\item any non-constant $J$-holomorphic disk with corners and boundary on the immersion $L_1 \sqcup L_2 \looparrowright M$ is either simple or multiply covered,
	\item any simple $J$-holomorphic disk with corners and boundary on the immersion $L_1 \sqcup L_2 \looparrowright M$ is regular (meaning that the linearization of the Cauchy-Riemann operator is surjective).
\end{enumerate}

\subsubsection{Surgery and count of holomorphic disks} 
We expect that we can apply our work to prove the following theorem.

\begin{Theo}[$\star$] \footnote{We use the symbol $\star$ to indicate the results that are work in progress.} 
\label{Theo:bijection strips et teardrops}
Let $(\varepsilon_\nu)_{\nu \in \N}$ be a sequence of positive real numbers such that $\varepsilon_\nu \to 0$. There is a second category subset $\mathcal{J}_{\phi,\text{reg},2}(M,\omega) \subset \mathcal{J}_{\phi, \text{reg}}(M,\omega)$ and $\nu_0 \in \N$ such that the following holds.

For any $\nu \geqslant \nu_0$, $J \in \mathcal{J}_{\phi,\text{reg},2}(M,\omega) $ and $y \in L_1 \cap L_2$ with $\norm{y} = 2$, there is a bijection $\mathcal{M}(y, J, L_1 \#_{x,\varepsilon_\nu} L_2) \to \mathcal{M}(y,x,J,L_1,L_2)$ between the set of $J$-holomorphic teardrops with boundary on $L_1 \#_{x,\varepsilon_\nu} L_2$ and the set of strips between $L_1$ and $L_2$ from $y$ to $x$.   
\end{Theo}

Below, we list the main steps expected to lead to this result.

\subsubsection{Surgery and holomorphic disks}
The proof of theorem \ref{Theo:bijection strips et teardrops} relies on three results. The first is a result about the multiplicity \footnote{See Remark \ref{rk:definition of the multiplicity} for the definition of multiplicity.} of an isolated generic $J$-holomorphic strip at its corners. 

\begin{prop}
	\label{prop : multiplicity of $J$-holomorphic strips}
There is a second category subset $\mathcal{J}_{\phi,\text{reg},3}(M, L_1, L_2, \omega) \subset \mathcal{J}_{\phi, \text{reg}}(M,\omega)$ such that if $J \in \mathcal{J}_{\phi,\text{reg},3}(M, L_1, L_2, \omega) $, then every $J$-holomorphic strip of Fredholm index $1$ has multiplicity $1$ at his corners. 		
\end{prop}

\begin{proof}[Sketch of the proof]

Fix $x$ and $y$ such that $\norm{x} - \norm{y} = 1$. Consider the universal moduli space $\mathcal{M}^{*} \left (L_1,L_2,x,y,\mathcal{J}^l \right )$ of pairs $(u,J)$ with 
\begin{itemize}[label = $\bullet$]
	\item $J$ a $\mathcal{C}^l$ almost complex structure in $\mathcal{J}_\phi(M,\omega)$,
	\item $u$ a $J$-holomorphic strip between $L_1$ and $L_2$ from $x$ to $y$.	
\end{itemize}

The usual arguments (as in \cite[Chapter 3]{McDS12}) show that $\mathcal{M}^{*} \left (L_1,L_2,x,y,\mathcal{J}^l \right )$ admits the structure of a smooth separable Banach manifold.

Assume that $(u,J) \in \mathcal{M}^{*} \left (L_1,L_2,x,y,\mathcal{J}^l \right )$, then (see Proposition \ref{prop:asymptoticdevelopmentaroundacornerpoint})\footnote{However, in this setting there is no real necessity to use \cite{RS01}. Since the curves are holomorphic near the double points, we can use the Schwarz reflection principle twice. } the limits 
	\[ \ev_{\text{x,jet}} (u,J) : =  \lim_{s \to - \infty} e^{-\frac{\pi}{2} s } u(s,t) ,\]
	and
	\[ \ev_{\text{y,jet}} (u,J) : =  \lim_{s \to + \infty} e^{+\frac{\pi}{2} s } u(s,t)  \]
exist. This defines two smooth maps $\ev_{\text{x,jet}} : \mathcal{M}^{*} \left (x,y,L_1,L_2,\mathcal{J} \right ) \to \R^{n}$ and $\ev_{\text{y,jet}} : \mathcal{M}^{*} \left (x,y,L_1,L_2,\mathcal{J} \right ) \to \R^{n}$. Notice that $\ev_{\text{x,jet}}^{-1}(0)$ (resp. $\ev_{\text{y,jet}}^{-1}(0)$ ) is the set of $J$-holomorphic strip with multiplicity greater than $1$ at $x$ (resp. $y$).

A variation of the arguments of \cite[3.4]{McDS12} show that these are submersions. Hence the sets $\ev_{\text{x,jet}}^{-1}(0)$ and $\ev_{\text{y,jet}}^{-1}(0)$ are smooth submanifolds of codimension $n$.

Now, one can see from the Sard-Smale theorem and an argument due to Taubes (see \cite[3.2]{McDS12} or \cite[Section 5]{FHS95}) that there is a generic subset $\tilde{\mathcal{J}} \subset \mathcal{J,\phi}(M,\omega)$ satisfying the following. For each $J \in \tilde{\mathcal{J}}$, $\ev_{\text{x,jet}}^{-1}(0) \cap \mathcal{M}^*(x,y,L_1,L_2,J)$ is a submanifold of codimension $n$ in $\mathcal{M}^*(x,y,L_1,L_2,J)$ which has dimension $1$. It is therefore empty (since $n \geqslant 3$).

The conclusion follows since any $J$-holomorphic strip is simple (cf Corollary \ref{coro:generic strips are simple}).
\end{proof}

In complex dimension greater than $3$, the same conclusion holds for teardrops.

\begin{prop}
	\label{prop : multiplicity of $J$-holomorphic teardrops}
There is a second category subset $\mathcal{J}_{\phi,\text{reg},4}(M, L_1, L_2, \omega)$ such that if $ J \in \mathcal{J}_{\phi,\text{reg},4}(M, L_1, L_2, \omega) $, then every $J$-holomorphic teardrop of Fredholm index $2$ and boundary on $L_1 \#_{x,\varepsilon_\nu} L_2$ for $\nu \geqslant 0$ has multiplicity $1$ at its corner. 		
\end{prop}

\begin{proof}
Fix $\nu \in \N$, as in the proof of Proposition \ref{prop : multiplicity of $J$-holomorphic strips}, there is a second category subset $\mathcal{J}^\nu_{\phi,\text{reg}}(M,\omega)$ such that every $J$-holomorphic teardrop with boundary on $L_1 \#_{x,\varepsilon_\nu} L_2$ has multiplicity $1$ at its corner.

Now the countable intersection 
\[ \mathcal{J}_{\phi,\text{reg},3}(M, L_1, L_2, \omega) := \bigcap_{\nu \geqslant 0} \mathcal{J}^\nu_{\phi,\text{reg}}(M,\omega) \] is of second category and satisfies the conclusion of the theorem.  	
\end{proof}

Consider an $\alpha > 0$, a complex structure $J \in \mathcal{J}_{\phi,\text{reg},3}(M, L_1, L_2, \omega) \cap \mathcal{J}_{\phi,\text{reg},4}(M, L_1, L_2, \omega)$ and an intersection point $y \in L_1 \cap L_2$ with $\norm{y} = 2$. Let $\mathcal{M}(y,L_1\#_{\varepsilon_\nu,x} L_2, J,\alpha) \subset \mathcal{M}(y,L_1\#_{\varepsilon_\nu,x} L_2,J)$ be the set of elements of $\mathcal{M}(y,L_1\#_{\varepsilon_\nu,x} L_2,J)$ represented by a $u \in \tilde{\mathcal{M}}(y,L_1\#_{\varepsilon_\nu,x} L_2,J)$ such that there is a strip $w \in \tilde{\mathcal{M}}(y,x,L_1,L_2,J)$ with 
 \[ \sup_{z} d_J(u(z),w(z)) < \alpha. \]
Here, $d_J$ is the distance induced by the metric $g_J$.

By the propositions above, there is a second category subset $\mathcal{J}_{\text{reg},5}(M,\omega, L_1,L_2)$ such that for $J \in \mathcal{J}_{\text{reg},5}(M,\omega, L_1,L_2)$
\begin{itemize}[label= $\bullet$]
	\item Every $J$-holomorphic strip in $\mathcal{M}(y,x,L_1,L_2,J)$ for $y \in L_1 \cap L_2$ with $\norm{y} = 2$ is regular, simple and has corners of multiplicity $1$,
	\item every $J$-holomorphic teardrop	 in $\mathcal{M}(y,L_1 \#_{\varepsilon,\nu}L_2,J)$ for $\nu \geqslant 0$ is regular, simple and has a corner of multiplicity $1$. 
\end{itemize}

Since the Lagrangians $L_1$ and $L_2$ are exact, Gromov compactness for $J$-holomorphic strips and regularity imply that the space $\mathcal{M}(y,x,L_1,L_2,J)$ is compact.

Therefore, we can apply a result stated in \cite[Theorem 5.11]{FOOO07} to obtain the following corollary.

\begin{coro}[$\star$]
\label{Coro: The theorem of FOOO}
There is a second category subset $\mathcal{J}_{\phi,\text{reg},5}(M, L_1,L_2,\omega)$ such that the following holds.

For every $J \in \mathcal{J}_{\phi,\text{reg},5}(M,\omega, L_1, L_2)$, there exist $\alpha > 0$, $\nu_0 \geqslant 0$ such that for $\nu \geqslant \nu_0$ there is a bijection
	\[\mathcal{M}(y,L_1 \#_{\varepsilon_\nu,x} L_2,J,\alpha)  \to \mathcal{M}(y,x,L_1,L_2,J)  .\]
\end{coro}

\subsubsection{Gromov compactness} 

Last we need a version of Gromov compactness for $J$-holomorphic curves as the surgery parameter $\varepsilon_\nu$ goes to $0$. We emphasize that it is not (to our knowledge) proved in the literature and that it is the subject of future work.

A \emph{$d$-leafed tree} is a planar tree $T \subset \R^2$ with a choice of vertex $\alpha$ called the root, oriented so that the root has no incoming edge and with $d$ leaves (beware that it is \emph{not} the definition of \cite[(9d)]{Sei08} ). For each vertex $v$ of $T$, we denote by $\norm{v}$ its valency.

\begin{defi}
A labeled domain consists of
\begin{enumerate}[label=(\roman*)]
	\item a $d$-leafed tree $T$,
	\item for each vertex $v$, an element $r_v \in \mathcal{R}^{\norm{v}}$,
	\item for each vertex $v$, $k_v \in \N$ cyclically ordered marked points at the boundary that we will denote by $z_1, \ldots, z_{k_v}$,
	\item for each connected component $C$ of $\partial r_v \backslash \{ z_1, \ldots, z_{k_v} \}$, an element $L_C \in \{ L_1, L_2 \}$,	
\end{enumerate}
which satisfy the following conditions. 
\begin{enumerate}
	\item If $C_1$ and $C_2$ are two adjacent connected components, then the labels $L_{C_1}$ and $L_{C_2}$ should be different,
	\item for every leaf $v$, $k_v \geqslant 1$.
\end{enumerate}	
\end{defi}

\begin{figure}
\captionsetup{justification=centering,margin=2cm}

\begin{tikzpicture}
\draw[directed] (2,2) -- (1,1);
\draw[directed] (1,1) -- (0,0);
\draw[directed] (1,1) -- (1,0);
\draw[directed] (1,1) -- (2,0);
\draw[directed] (2,2) -- (4,0);

\draw (0,0) node {\tiny $\bullet$};
\draw (1,0) node {\tiny $\bullet$};
\draw (2,0) node {\tiny $\bullet$};
\draw (1,1) node {\tiny $\bullet$};
\draw (2,2) node {\tiny $\bullet$};
\draw (4,0) node {\tiny $\bullet$};

\draw (2,2) node[above] {\tiny $\alpha$};

\draw[blue] (9,2.5) arc (90:220:0.5);
\draw[red] (9,2.5) arc (90:-30:0.5);
\draw[red] (9,1.5) arc (-90:-150:0.5);
\draw[blue] (9,1.5) arc (-90:-30:0.5);

\draw[red] (8.57,1.75) arc (30:390:0.5);
\draw[blue] (8.57,1.75) arc (30:330:0.5);
\draw[red] (8.57,1.75) arc (30:270:0.5);
\draw[blue] (8.57,1.75) arc (30:210:0.5);

\draw[blue] (7.71,1.25) arc (30:210:0.25);
\draw[red] (7.71,1.25) arc (30:-150:0.25);

\draw[blue] (8.14,1) arc (90:-90:0.25);
\draw[red] (8.14,1) arc (90:270:0.25);

\draw[blue] (8.57,1.25) arc (150:330:0.25);
\draw[red] (8.57,1.25) arc (150:-30:0.25);

\draw[blue] (9.43,1.75) arc (150:330:0.5);
\draw[red] (9.43,1.75) arc (150:-30:0.5); 

\draw (9,2.5) node[above] {\tiny $y$};
\draw (9,1.5) node {\tiny $\bullet$};

\draw (9,1) node {\tiny $\bullet$};
\draw (8.14,0.5) node {\tiny $\bullet$};
\draw (7.28,1) node {\tiny $\bullet$};

\draw (10.29,1.25) node {\tiny $\bullet$};

\end{tikzpicture}

\caption{A labeled domain and its underlying tree \\
Red corresponds to a label $L_1$ and blue to $L_2$ \\
The dots are mapped to $x$}
\label{Figure: Labeled domain}	
\end{figure}
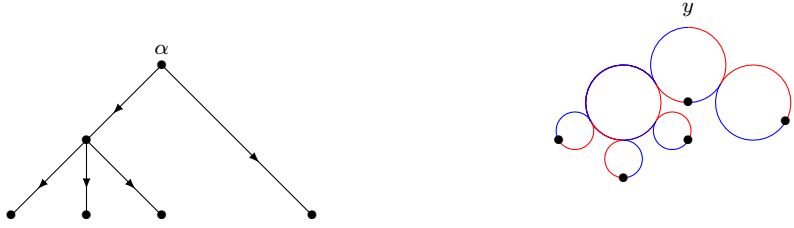

\begin{rk}
For each vertex $v$ one can number the outgoing edges as counterclockwise (remember that $T$ is embedded in $\R^2$). The number of an edge $e$ going from $v_1$ to $v_2$ will be denoted $n_e$. 	
\end{rk}
 
This leads us to the definition of the limit curves when the handle parameter goes to $0$.

\begin{defi}
A broken strip from $y$ to $x$ modeled on the labeled domain $T$ consists of 
\begin{enumerate}[label=(\roman*)]
	\item for each vertex $v$, $\norm{v} + 1$ intersection points $y_0^v, \ldots , y_{\norm{v}}^v \in L_1 \cap L_2$ such that if $e$ is an edge from $v_1$ to $v_2$ then $y^{v_1}_{n_e} = y^{v_2}_0$,
	\item for each vertex $v$ a $J$-holomorphic curve with corners $u_v : (r_v, \partial r_v) \to (M, L_1 \sqcup L_2)$,
\end{enumerate}
such that

\begin{enumerate}
	\item for each vertex $v$ and connected component $C$ of $\partial r_v \backslash \{z_1, \ldots, z_{k_v} \}$, we have $u_v(C) \subset L_C$,
	\item for each vertex $v$ and $i \in \{ 1, \ldots , k_v\}$, we have $v(z_i) = x$,
	\item for the root vertex $\alpha$, $u_\alpha$ converges to $y$ on the $0$-th strip-like end of $r_\alpha$,
	\item for each vertex $v$, the curve $u$ converges to $y^v_i$ on the $i$-th strip-like end of $r_v$. 	
\end{enumerate}
\end{defi}

We expect that the following proposition is an adaptation of the neck-stretching procedure (as it appears in \cite{BEHWZ} and \cite{CM05}) for curves with boundary on a Lagrangian manifold.

\begin{prop}[$\star$]
\label{prop:Gromov compacity for handles}
Fix an almost complex compatible structure $J \in \mathcal{J}_{\phi}(M,\omega)$.

Assume that $(u_\nu)_{\nu \in \N}$ is a sequence of $J$-holomorphic teardrops such that 
\[ u_\nu \in \mathcal{M}(y, L_1 \#_{x, \varepsilon_\nu} L_2, J ), \forall \nu \geqslant 0 .\]
There is a subsequence $(u_{\nu_k})_{k \geqslant 0}$ which Gromov converges to a broken strip from $y$ to $x$. 
\end{prop}
    
\subsubsection{Proof of Theorem \ref{Theo:bijection strips et teardrops} ($\star$)}
We now prove Theorem \ref{Theo:bijection strips et teardrops} assuming that Proposition \ref{prop:Gromov compacity for handles} is true. For this fix $J \in \mathcal{J}_{\phi, \text{reg},5}(M,\omega)$.

Given the conclusion of Corollary \ref{Coro: The theorem of FOOO}, it only remains to check that there is $\nu_0$ such that for all $\nu \geqslant \nu_0$, any teardrop $v \in \tilde{\mathcal{M}}(y,J,L_1 \#_{x,\varepsilon_\nu} L_2)$ is $\alpha$ close to a $J$-holomorphic strip $v \in \tilde{\mathcal{M}}(y,x,J,L_1,L_2)$.

Assume by contradiction that there is a strictly increasing sequence $(\nu_k)$, and a sequence of teardrops $u_{\nu_k} \in \tilde{\mathcal{M}}(y,J,L_1 \#_{x,\varepsilon_{\nu_k}} L_2)$ such that
\[ \forall v \in  \tilde{\mathcal{M}}(y,x,J,L_1,L_2), \ \sup_{z \in \D} d_J (u(z),v(z)) \geqslant \alpha . \]

By Proposition \ref{prop:Gromov compacity for handles}, there is a subsequence of $(u_{\nu_k})$ which converges in the sense of Gromov to a broken strip $w$.

To conclude, it remains to see that $v$ is an actual teardrop. This is immediate from the two lemmas below.

\begin{lemma}
Assume $w = (u_v)_{v \in T}$ is a broken strip with underlying tree $T$ such that all $u_v$ are simple. Then the tree $T$ consists of one vertex and $w$ is a strip from $y$ to $x$.	
\end{lemma}

\begin{proof}
	This is a simple combinatorial argument which uses regularity and simplicity of the underlying holomorphic curves.
	
First, notice that the index of $x$ as an element of $CF(L_2,L_1)$ is $n-1$ which is greater than $1$ since $n \geqslant 3$.
	
For $v \in T$ different from the root, call $y_v$ the incoming limit point and $x_1, \ldots x_p$ the outgoing limit points. Moreover, assume that there are $k_{1,v}$ marked points mapping to $x$ going from $L_1$ to $L_2$ and $k_{2,v}$ marked points mapping to $x$ going from $L_2$ to $L_1$. Since the curve $u_v$ is regular, we have
\[ \norm{y_v} - \sum_{i=1}^p \norm{x_i} - k_{1,v} - k_{2,v}(n-1) + k_{1,v} + k_{2,v} + \norm{v} - 3 \geqslant 0 ,\]
so
\[ \norm{y_v} - \sum_{i=1}^p \norm{x_i} + \norm{v} - 3 \geqslant 0 . \]

Similarly, if $v$ is the root, we get
\[ \norm{y} - \sum_{i=1}^p \norm{x_i} - k_{1,v} - k_{2,v}(n-1) + k_{1,v} + k_{2,v} + \norm{v} + 1 - 3 \geqslant 0, \]
so
\[ \norm{y} - \sum_{i=1}^p \norm{x_i} + \norm{v} - 2 \geqslant 0. \]  

Adding these equalities for $v \in T$, we obtain
\[ 	\norm{y} + \sum_{v \in T} \norm{v} - 3 V(T) + 1 \geqslant 0 ,\]
where $V(T)$ is the number of vertices of $T$. Now notice that $\sum_{v \in T} \norm{v}$ is twice the number of edges of $T$ and therefore equal to $2 V(T) - 2$. So
\[ 2 = \norm{y} \geqslant 1 + V(T). \]
Hence $V(T) = 1$. Therefore we have a single curve $w$ with one corner at $y$ and the others at $x$.

Now since $y$ is an incoming point from $L_1$ to $L_2$, there are $2k-1$ other corners mapping to $x$ (with $k$ an integer greater than $1$). Among them, $k$ are outgoing points from $L_1$ to $L_2$ and $k-1$ are outgoing points from $L_2$ to $L_1$. Since $w$ is regular, we get
\[ \norm{y} -k -(n-1)(k-1) + 2k -3 \geqslant 0, \]
so
\[ \norm{y} - (n-2)k + n-4 \geqslant 0, \]
hence (since $\norm{y}  = 2$)
\[ n-2 \geqslant (n-2)k. \]
Since $n-2 \geqslant 1$, we readily conclude that $1 \geqslant k$ hence $k = 1$.
\end{proof}

\begin{lemma}
Assume $w = (u_v)_{v \in T}$ is a broken strip with underlying tree $T$. There is a tree $T_1$ and a broken strip $w_1 = (u_{v,1})_{v \in T}$ with underlying tree $T_1$ such that the following assertions hold.
\begin{enumerate}
	\item For each $v \in T_1$, the curve $u_{v,1}$ is simple.
	\item There is an injective tree morphism $f : T_1 \to T$ mapping the root of $T_1$ to the root of $T$ satisfying the following. If $v \in T_1$, the underlying simple curve of $u_{f(v)}$ is $u_{v,1}$.
	\item If $V(T) \geqslant 2$, then $V(T_1) \geqslant 2$. 	
\end{enumerate} 	
\end{lemma}

\begin{proof}
The simple curve is built by an induction process.

Start with the root $v_0$. The curve $u_{v_0}$ is multiply covered by the choice of almost complex structure $J$. Let $u_{v_0,1}$ be the underlying simple curve : there is a branched cover $\pi$ such that $u_{v_0} =  u_{v_0,1} \circ \pi $. We associate the curve $u_{v_0,1}$ to the root of $T_1$.

The domain of $u_{v_0,1}$ has one incoming strip-like end (the image of the incoming strip-like end of $r_{v_0}$ by $\pi$) and $m_v \in \N$ outgoing strip-like ends. Call $\zeta_1, \ldots, \zeta_{m_v}$ their asymptotic points. For each $\zeta_i$ we put an outgoing edge $e_{\zeta_i}$. Call $v_{\zeta_i}$ the outgoing end of $e_{\zeta_i}$.

For each $i\in \{1, \ldots , m_v \}$, choose a point $\tilde{\zeta}_i \in r_{v_0}$ such that $\pi (\tilde{\zeta}_i) = \zeta_i$. Each $\tilde{\zeta}_i$ is the limit of an outgoing strip-like ends and 	corresponds to an edge in $T$ with endpoint $v_{\tilde{\zeta}_i}$. The curve $u_{v_{\zeta_i}}$ is the simple curve underlying $u_{v_{\tilde{\zeta}_i}}$. 

If we repeat this process by induction, it is easy that the end-product is a broken strip satisfying the hypotheses.   
\end{proof}

\bibliographystyle{alpha}
\bibliography{Bibliographie}

\newcommand{\etalchar}[1]{$^{#1}$}
\begin{thebibliography}{FOOO06}

\bibitem[AJ10]{AJ10}
Manabu Akaho and Dominic Joyce.
\newblock Immersed {L}agrangian {F}loer theory.
\newblock {\em J. Differential Geom.}, 86(3):381--500, 2010.

\bibitem[BC07]{BC07}
Paul {Biran} and Octav {Cornea}.
\newblock {Quantum Structures for Lagrangian Submanifolds}.
\newblock {\em arXiv:0708.4221}, August 2007.

\bibitem[BC09]{BC09}
Paul Biran and Octav Cornea.
\newblock Rigidity and uniruling for {L}agrangian submanifolds.
\newblock {\em Geom. Topol.}, 13(5):2881--2989, 2009.

\bibitem[BC13]{BC13}
Paul Biran and Octav Cornea.
\newblock Lagrangian cobordism. {I}.
\newblock {\em J. Amer. Math. Soc.}, 26(2):295--340, 2013.

\bibitem[BEH{\etalchar{+}}03]{BEHWZ}
Frederic Bourgois, Yakov Eliashberg, Helmut Hofer, Kris Wisocki, and Eduard
  Zehnder.
\newblock Compactness results in symplectic field theory.
\newblock {\em Geom. Topol.}, 7:799--888, 2003.

\bibitem[CM05]{CM05}
Kai Cieliebak and Klaus Mohnke.
\newblock Compactness for punctured holomorphic curves.
\newblock {\em J. Symplectic Geom.}, 3(4):589--654, 2005.

\bibitem[FHS95]{FHS95}
Andreas Floer, Helmut Hofer, and Dietmar Salamon.
\newblock Transversality in elliptic {M}orse theory for the symplectic action.
\newblock {\em Duke Math. J.}, 80(1):251--292, 1995.

\bibitem[FOOO06]{FOOO07}
Kenji Fukaya, Yong-Geun Oh, Hiroshi Ohta, and Kaoru Ono.
\newblock Lagrangian surgery and holomorphic discs.
\newblock 2006.

\bibitem[KO00]{KO00}
Daesung Kwon and Yong-Geun Oh.
\newblock Structure of the image of (pseudo)-holomorphic discs with totally
  real boundary condition.
\newblock {\em Comm. Anal. Geom.}, 8(1):31--82, 2000.

\bibitem[Laz00]{La00}
Laurent Lazzarini.
\newblock Existence of a somewhere injective pseudo-holomorphic disc.
\newblock {\em Geom. Funct. Anal.}, 10(4):829--862, 2000.

\bibitem[Laz11]{La11}
Laurent Lazzarini.
\newblock Relative frames on {$J$}-holomorphic curves.
\newblock {\em J. Fixed Point Theory Appl.}, 9(2):213--256, 2011.

\bibitem[MS12]{McDS12}
Dusa McDuff and Dietmar Salamon.
\newblock {\em {$J$}-holomorphic curves and symplectic topology}, volume~52 of
  {\em American Mathematical Society Colloquium Publications}.
\newblock American Mathematical Society, Providence, RI, second edition, 2012.

\bibitem[Oh93a]{Oh93}
Yong-Geun Oh.
\newblock Floer cohomology of {L}agrangian intersections and pseudo-holomorphic
  disks. {I}.
\newblock {\em Comm. Pure Appl. Math.}, 46(7):949--993, 1993.

\bibitem[Oh93b]{Oh93-2}
Yong-Geun Oh.
\newblock Floer cohomology of {L}agrangian intersections and pseudo-holomorphic
  disks. {II}{$({\bf C}{\rm P}^n,{\bf R}{\rm P}^n)$}.
\newblock {\em Comm. Pure Appl. Math.}, 46(7):995--1012, 1993.

\bibitem[Oh97]{Oh97}
Yong-Geun Oh.
\newblock On the structure of pseudo-holomorphic discs with totally real
  boundary conditions.
\newblock {\em J. Geom. Anal.}, 7(2):305--327, 1997.

\bibitem[RS01]{RS01}
Joel~W. Robbin and Dietmar~A. Salamon.
\newblock Asymptotic behaviour of holomorphic strips.
\newblock {\em Ann. Inst. H. Poincar\'e Anal. Non Lin\'eaire}, 18(5):573--612,
  2001.

\bibitem[Sei00]{Sei00}
Paul Seidel.
\newblock Graded {L}agrangian submanifolds.
\newblock {\em Bull. Soc. Math. France}, 128(1):103--149, 2000.

\bibitem[Sei08]{Sei08}
Paul Seidel.
\newblock {\em Fukaya categories and {P}icard-Lefschetz theory}.
\newblock Zurich {L}ectures in {A}dvanced {M}athematics. European Mathematical
  Society (EMS), Z\"urich, 2008.

\end{thebibliography}

\end{document}